\documentclass[a4paper,10pt]{amsart}
\usepackage[utf8]{inputenc}
\usepackage{amsmath, amssymb,bm}
\usepackage[all]{xy}
\title{Isotopes of Octonion Algebras, $\mathbf{G}_2$-Torsors and Triality}
\author[S.\ Alsaody, P.\ Gille]{Seidon Alsaody and Philippe Gille}
\address{Univ Lyon, Universit\'e Claude Bernard Lyon 1, CNRS UMR 5208, Institut Camille Jordan, 
43 blvd. du 11 novembre 1918, F-69622 Villeurbanne cedex, France.}
\thanks{The first author is supported by the grant KAW 2015.0367 from the Knut and Alice Wallenberg Foundation. 
The second author is supported by the project ANR Geolie, ANR-15-CE40-0012, (The French National Research Agency).}

\newtheorem{Thm}{Theorem}[section]
\newtheorem{Prp}[Thm]{Proposition}
\newtheorem{Cor}[Thm]{Corollary}
\newtheorem{Lma}[Thm]{Lemma} 
\theoremstyle{definition}

\newtheorem{Rk}[Thm]{Remark}
\newtheorem{Ex}[Thm]{Example}

\numberwithin{equation}{section}

\def\ot{\otimes}
\newcommand\co{\colon}
\newcommand\chv{^{\scriptscriptstyle\vee}}

\newcommand{\rad}{\operatorname{rad}}
\newcommand{\Pic}{\operatorname{Pic}}
\newcommand{\Ralg}{R\mathchar45\mathbf{alg}}

\newcommand{\Hom}{\operatorname{Hom}}

\newcommand{\GL}{\mathrm{GL}}
\newcommand{\SO}{\mathbf{SO}}

\newcommand{\Cl}{\mathrm{Cl}}
\newcommand{\fppf}{\mathrm{fppf}}
\newcommand{\Ker}{\mathrm{Ker}}
\newcommand{\End}{\mathrm{End}}

\newcommand{\Z}{\mathbb{Z}}
\newcommand{\Q}{\mathbb{Q}}
\newcommand{\F}{\mathbb{F}}

\newcommand{\bS}{\mathbf{S}}
\newcommand{\ubS}{\underline{\mathbf{S}}}
\newcommand{\Id}{\mathrm{Id}}
\newcommand{\Lie}{\mathrm{Lie}}

\newcommand{\bAut}{\mathbf{Aut}}
\newcommand{\bGL}{\mathbf{GL}}
\newcommand{\RT}{\mathbf{RT}}
\newcommand{\Spin}{\mathbf{Spin}}
\newcommand{\Spec}{\mathrm{Spec}}
\newcommand{\Stab}{\mathrm{Stab}}
\newcommand{\simlgr}{\xrightarrow{\sim}}
\newcommand{\gX}{\mathfrak{X}}

\newcommand{\bH}{\mathbf{H}}
\newcommand{\bG}{\mathbf{G}}
\newcommand{\bY}{\mathbf{Y}}

\newcommand{\bE}{\mathbf{E}}

\newcommand{\calL}{\mathcal{L}}
\newcommand{\bmu}{\bm{\mu}}
\newcommand{\Disc}{\mathbf{Disc}}
\newcommand{\calM}{\mathcal{M}}

\entrymodifiers={+!!<0pt,\fontdimen22\textfont2>}
\begin{document}

\begin{abstract} Octonion algebras over rings are, in contrast to those over 
fields, not determined by their norm forms. 
Octonion algebras
whose norm is isometric to the norm $q$ of a given algebra $C$ are twisted forms 
of $C$ by means of the $\bAut(C)$--torsor 
$\mathbf{O}(C)\to\mathbf{O}(q)/\bAut(C)$.

We show that, over any commutative unital ring, these twisted forms are 
precisely the isotopes $C^{a,b}$ of $C$, with multiplication given by 
$x*y=(xa)(by)$, for unit norm octonions $a,b\in C$. The link is provided by the 
triality phenomenon, which we study from new and classical perspectives. We 
then study these twisted forms using the interplay, thus obtained, between 
torsor geometry and isotope computations, thus obtaining new results on octonion algebras over e.g.\ rings of (Laurent) polynomials. 
\end{abstract}

\maketitle

\bigskip


\medskip


\date{\today}

\section{Introduction}
The aim of this paper is to give a concrete construction, over unital 
commutative rings, of all pairwise 
non-isomorphic octonion algebras having isometric norm forms. To this end we develop the framework of triality 
over rings. We begin by outlining 
our philosophy and intuition.

Let $C$ be an octonion algebra over a unital, commutative ring $R$. In 
particular $C$ is equipped with
a multiplicative, regular quadratic form $q$. Then every automorphism of $C$ is 
an isometry with respect to $q$,
but the converse is far from being true. More precisely, the automorphism group
of $C$ is a 14-dimensional exceptional group scheme of type $\mathrm{G}_2$, 
while the orthogonal group of $C$
has dimension 28. Despite this, if $R$ is a field or, more generally, a local 
ring, then 
every octonion algebra whose quadratic form is isometric to $q$ is isomorphic, 
as an algebra, to $C$. 
This was proved false by the second author in \cite{G2} over more general rings, 
using torsors
and homotopy theory to study the homogeneous space $\SO(q)/\bAut(C)$.

This space measures, in a sense, the gap between the (special) orthogonal group 
of $q$ and the automorphism group. 
Our initial observation in the current work is that any
element of the spin group of $q$ induces an isomorphism 
from $C$ to a certain isotope $C^{a,b}$ with $a,b\in \bS_C(R)$,
the 7-dimensional unit sphere of $C$, that is the sphere for the octonionic 
norm on $C$. The algebra $C^{a,b}$ is simply the module 
$C$ with multiplication $x*y=(xa)(by)$.
As $a$ and $b$ run through $\bS_C(R)$, we exhaust the special orthogonal group 
modulo the automorphism group. As the match in
dimensions ($28-14=7+7$) roughly indicates, these isotopes essentially account 
for the full gap between isometries and automorphisms.

The key to this observation is the phenomenon of \emph{triality}. Well known 
over fields, it was extended to rings in \cite{KPS}. We give another formulation 
of
this phenomenon over rings, in a way more streamlined for our arguments. With 
this tool we are able to show that the isotopes above arise naturally as twists 
of $C$ by torsors
of the type considered in \cite{G2}. Moreover we show that these isotopes 
account for \emph{all} octonion algebras with isometric norms.

The paper is organized as follows. In Section \ref{sec_background}, we introduce 
the isotopes of a given octonion $R$--algebra and recall background information 
on octonions and their isotopes. In Section \ref{sec_triality}, we prove using 
scheme theoretic arguments that the spin
group of an octonion algebra over any unital commutative ring can be described 
in terms of triality and so called related triples. This enables us to give a 
precise description of the $\bAut(C)$--torsor 
$\Pi:\mathbf{Spin}(q)\to\mathbf{Spin}(q)/\bAut(C)$. For the sake of completeness 
we also give a concrete construction of the spin-triality correspondence by 
extending arguments valid over fields in a slightly different way than was done 
in \cite{KPS}. 

Our main results are found in Sections \ref{sec_twist} and \ref{sec_variants}. 
One the one hand, we show that the $\Pi$--twists of $C$ correspond canonically 
to the isotopes $C^{a,b}$ (Theorem \ref{thm_main}). On the other hand, in 
Theorem \ref{thm_kernel}, we show that all octonion algebras with norm isometric 
to that of $C$ are obtained in this way. This is done by showing that the torsor 
$\Pi$ gives the same objects as the $\mathbf{O}(q)\to\mathbf{O}(q)/\bAut(C)$ 
(i.e.\ the corresponding cohomology kernels coincide). In the same theorem we 
obtain similar coincidences for a number of other torsors. We then relate our 
results to that of so-called compositions of quadratic forms.

Finally in Section \ref{sec_particular} we discuss interesting special cases, 
such as rings over which all isotopes are isomorphic, and isotopes which are 
isomorphic over all rings.

We fix, once and for all, a unital commutative ring $R$. By an \emph{$R$--ring} 
we mean a unital, commutative and associative $R$--algebra. 

\subsection{Acknowledgements}
We are grateful to E.\ Neher and K.\ Zainoulline for fruitful discussions. To E.\ Neher we are also indebted for his thorough reading of an earlier version and his remarks that helped improve the paper.
We would also like to thank A.\ Asok, M.\ Hoyois and M.\ Wendt
for communicating their preprint \cite{AHW} to us.

\section{Background}\label{sec_background}
\subsection{Octonion algebras and isotopy}
Isotopy of algebras is a very general notion, which offers a way of deforming 
the multiplicative structure of an algebra. The notion goes back to Albert 
\cite{Al}. We will outline it here in order to derive, from its generality, the 
proper context needed for our purposes in 
a self contained manner. The contents of this subsection can otherwise be found 
in \cite{McC}.

An \emph{algebra} $A=(A,*)=(A,*_A)$ over $R$ is an $R$--module $A$ endowed with 
an $R$--bilinear multiplication $*$. In this generality, neither associativity 
nor commutativity nor the existence of
a unity is required. For each $a\in A$ we write $L_a$ or $L_a^*$ for the linear 
map $x\mapsto a*x$ on $A$, and $R_a$ or $R_a^*$ for the linear map $x\mapsto 
x*a$. If $A$ is alternative, then $L_aR_a=R_aL_a$ for any $a\in A$,
and we denote this map by $B_a$. 

An \emph{octonion algebra} is an $R$-algebra whose underlying module is projective of constant rank 8, and which is endowed with a regular multiplicative quadratic form. (See 
also \cite[\S 4]{LPR} for further discussion.) Equivalently, it is a composition algebra of constant rank 8. (An algebra $A$ is a \emph{composition algebra} 
if the underlying module is finitely generated, faithfully projective, and if 
$A$ is endowed with a non-singular multiplicative quadratic form $q=q_A$, referred to as the 
\emph{norm} of $A$. \emph{We moreover always require that composition algebras 
be unital.}) It is known that locally, composition algebras have rank 1, 2, 4 or 
8. 

Two arbitrary $R$--algebras $(A,*_A)$ and $(B,*_B)$ said to be \emph{isotopic} 
or \emph{isotopes of one another} if there exist invertible linear maps 
$f_i:A\to B$, $i=1,2,3$, such that
\[f_1(x*_Ay)=f_2(x)*_Bf_3(y)\]
for all $x,y\in A$. The triple $(f_1,f_2,f_3)$ is then called an \emph{isotopy 
from $A$ to $B$}. It is an easy known fact that isotopy is an equivalence 
relation. Then $A$ and $B$ are isomorphic as $R$--modules, but not necessarily 
as $R$--algebras. However, the map $f_1: A\to B$ is an isomorphism of algebras 
between 
$A$ and $(B,*_B')$, where the new multiplication on $B$ is defined by
\[x*_B'y=f_2f_1^{-1}(x)*_Bf_3f_1^{-1}(y).\]

An algebra $(B,*')$ is said to be a \emph{principal isotope} of $(B,*)$ if there 
exist invertible linear maps $g,h:B\to B$ such that $x*'y=g(x)*h(y)$ holds for 
all $x,y\in B$. Then we denote $(B,*')$ by $B_{g,h}$ and note that
$B$ and $B_{g,h}$ are isotopic. Conversely, the previous paragraph implies that 
two algebras $A$ and $B$ is isotopic if and only if $A$ is isomorphic to a 
principal isotope of $B$. Thus
restricting to principal isotopes induces no loss of generality

In this paper, we are mainly concerned with unital algebras. Let $(A,*)$ be an 
$R$--algebra and let $g,h$ be invertible maps on $A$.
Then $A_{g,h}$ is unital precisely when there exists $e\in A$ such that
\[g(x)*h(e)=g(e)*h(x)=x\]
for all $x\in A$. This is equivalent to the condition
\[\begin{array}{lll}g^{-1}=R_{h(e)}^*&\text{and}&h^{-1}=L_{g(e)}^*.\end{array}\]
If $A$ is moreover a unital alternative algebra, then by \cite[Proposition 
2]{McC} the above implies that 
$g(e)$ and $h(e)$ are invertible, and denoting the inverses by $b$ and $a$, 
respectively, the above
condition becomes
\[\begin{array}{lll}g=R_{a}^*&\text{and}&h=L_{b}^*.\end{array}\]
Conversely, for any $a,b\in A^*$, the algebra $A_{R_a,L_b}$ is unital with unity 
$(ab)^{-1}$. 
In fact \cite{McC} shows the following result.

\begin{Prp} Let $C$ be a unital alternative algebra over $R$ and let $C'$ be 
isotopic to $C$. Then $C'$ is unital 
if and only if $C'\simeq C_{R_a,L_b}$ for some $a,b\in C^*$.
\end{Prp}

Here $C^*$ denotes the set of invertible elements of $C$. If $C$ is a composition algebra, then we have the equality $C^*=\{x\in C|q_C(x)=0\}$.

To lighten notation, we will henceforth write $A^{a,b}$ instead of 
$A_{R_a,L_b}$.
\begin{Rk}\label{Rbase} This construction is stable under base change, i.e.\ if 
$\phi:R\to S$ 
is a map of $R$--rings and $a,b\in\bS_C(R)$, then 
$(C^{a,b})_S=(C_S)^{\phi(a),\phi(b)}$ follows from the 
bilinearity of the multiplication. Therefore,
we may write $C_S^{a,b}$ without ambiguity.
\end{Rk}

\begin{Rk}\label{Rcomp} If $C$ is a composition algebra with quadratic form $q$, 
then $C'=C^{a,b}$ is a composition algebra 
as well, with norm $q'=\lambda q$, where $\lambda=q(a*_C b)$.  Since $q \circ 
L_{a*_C b}=\lambda q=q'$, 
it follows that  $(C, q_C)$ and $(C',q_{C'})$ are isometric quadratic spaces. 
Note that this isometry 
maps $1_C$ to $1_{C'}$. 
\end{Rk}

\subsection{Relations between isotopes}\label{subsec_relation}
As the previous subsection shows, if $C$ is an octonion algebra, then any unital 
algebra isotopic to $C$
is again an octonion algebra and is isomorphic to $C^{a,b}$ for some $a,b\in 
C^*$. We will simplify the presentation
further by exhibiting some basic isomorphisms between certain isotopes. This is 
the content of the following result.

\begin{Prp}\label{prop_formulae} Let $C$ be a unital alternative $R$--algebra 
and let $a,b\in C^*$.
\begin{enumerate}
\item In the commutative diagram
\[\xymatrix@1{
&C^{a,b}\\
C^{1,aba} \ar[ur]^{L_a}\ar[rr]_{R_{b^{-1}}L_a}&& C^{bab,1}\ar[ul]_{R_b} 
}\]
each arrow is an isomorphism of algebras.

\item In the commutative diagram
\[\xymatrix@1{
&C^{a,a^{-1}}\ar[dr]^{R_a}\\
C^{1,a}\ar[ur]^{L_a}\ar[rr]_{B_a}&&C^{a^{-1},1}}\]
each arrow is an isomorphism of algebras.

\item The maps $B_a:C^{a,b}\to C^{1,b a^{-1}}$ and $B_b: C^{a,b}\to 
C^{b^{-1}a,1}$ are isomorphisms of algebras. 
In particular $C^{a,a}\simeq C$.

\item The maps $B_{ba^{-1}}B_a:C^{a,b}\to C^{a b^{-1},1}$ and 
$B_{b^{-1}a}B_b:C^{a,b}\to C^{1,a^{-1}b}$ are isomorphisms of algebras.
\end{enumerate}
\end{Prp}

\begin{proof} All map  are clearly invertible and (1) was established by 
McCrimmon 
\cite[prop. 6]{McC}. Assertion (2) is the special case $b=a^{-1}$ of (1). For, 
(3) We check the first fact, the second being similar. Given $x, y \in C$, we 
have
\begin{eqnarray} \nonumber
 B_a( x *_{a,b} y)&= & B_a( (xa)(by)) \\   \nonumber
 &= & (a(xa)) \, ((by)a) \quad \hbox{[middle  Moufang law]} \\   \nonumber
 &= & (a(xa))  \, \bigl( ((ba^{-1})a)y) a  \bigr) \\   \nonumber
 &= & (a(xa)) \, \bigl(  (ba^{-1}) (a (ya)) \bigr) \quad \hbox{[one-sided  
Moufang law]}  \\ 
 &=& B_a(x) *_{1,ba^{-1}} B_a(y).  \nonumber
\end{eqnarray}

Finally, (4) follows by combining (2) and (3).
\vskip-6mm
\end{proof}

\begin{Rk}\label{rem_KPS} Knus--Parimala--Sridharan introduced another kind of 
isotope \cite[Rem. 4.7]{KPS}.
For an $R$--alternative algebra $C$ and an element $u \in C^*$, they defined the 
new multiplication
$x \star_{u} y = ( x (y u)) u^{-1}$ on $C$. We leave it to the reader to check 
that the 
map $R_{u}: C\to C$ induces an isomorphism of $R$--algebras $R_u: (C,\star_{u}) 
\simlgr (C ,*_{u^{-1},1})= C^{u^{-1},1} \simeq 
C^{1,u}$.
\end{Rk}

\begin{Rk}\label{Rreduce} The above proposition in fact implies that if $C$ is a 
composition algebra, then any 
composition algebra $C'$ isotopic to $C$ is isomorphic to e.g.\  $C^{1,a}$ with 
$q(a)=1$. Indeed, combining items (2) and (3) of the proposition, we have
\[C'\simeq C^{c^{-1},c} \simeq C^{1,c^2}.\]
As multiplication by
$q(c)$ is an isomorphism $C^{1,c^2}\to C^{1,q(c)^{-1}c^2}$, the claim follows. 
\end{Rk}

These relations will be explained partly by triality in Section 
\ref{subsec_rel_trial}.

\subsection{Direct summands of octonion algebras}\label{subsec_summand}

Let $C$ be an octonion algebra over $R$. As for Azumaya algebras, we recall the following fact, which is of course 
obvious if $2$ belongs to $R^*$. The \emph{trace map} $\mathrm{tr}:C\to R$ is defined by $\mathrm{tr}(x)=b_{q_C}(x,1)$.

\begin{Lma}\label{lem_trace} The trace map $\mathrm{tr} : C \to R$  admits a 
section.
\end{Lma}

\begin{proof} By means of partition of unity, we can assume that $R$ is local 
with maximal ideal 
$\mathfrak{M}$.  We know that $\mathrm{tr}(c)=b_{q}(1,c)$ for each $c \in C$. 
Since $b_{q}$ is a non-degenerate bilinear form,  there exists $c_K \in K=R/ 
\mathfrak{M}$ such that
$\mathrm{tr}(c_K) \not =0$. Let $c$ be a lift of $c_K$ in $C_{R/\mathfrak{M}}$. 
Then $\mathrm{tr}(c) \in R^*$ so that 
$\mathrm{tr}: C \to R$ is surjective. Thus  $\mathrm{tr} : C \to R$  admits a 
section.
\end{proof}

Let us recall also the next fact.

\begin{Lma}\label{lem_centre} We have
$$R1= \Bigl\{ x \in C \, \mid xy=yx \enskip \forall y \in C \Bigr\}
= \Bigl\{ x \in C \, \mid \, y(xz)=(yx)z \enskip \forall y,z \in C \Bigr\}.
$$
\end{Lma}

\begin{proof}
We prove it for  the right hand side (denoted  by $M$) and the middle term can be  done analogously.
Let $(c_i)_{i=1, \dots, n}$ be a  generating family of the $R$--module $C$. 
By bilinearity of the product, we have
$$
M=  \Bigl\{ x \in C \, \mid \, c_i(x c_j)=(c_i x) c_j \enskip \forall i,j =1, \dots, n  \Bigr\}.
$$
By using the classical trick to write $R$ as the limit of its finitely generated $\Z$-algebras, 
we may assume that $R$ is noetherian, so that $M$ is finitely presented.

Let  $\mathfrak{M}$ be a maximal ideal of $R$. 
By the corresponding result over fields \cite[Proposition 1.9.2]{SV}, we have 
$R/\mathfrak{M}= M \otimes_R (R/\mathfrak{M}) \simeq M/\mathfrak{M}M $.
Nakayama's lemma  shows that $R_\mathfrak{M}= M \otimes_R R_\mathfrak{M}$. 
Since it holds for all maximal ideals,  we conclude that 
$R=M$ as desired.
\end{proof}

For the next lemma we recall the identity $b_q(x,ay)=b_q(\overline{a}x,y)$, 
which implies that
\[\{y\in C \, | \, b_{q}(a y, x)=0\}=(\overline{a}x)^\bot\]
for all $a,x\in C$. Here $\bot$ denotes the orthogonal complement with respect 
to $b_q$.

\begin{Lma}\label{lem_summand} Let $x \in C$ be an invertible element and let $a 
\in C$ be an element of trace $1$.\footnote{Such an element exists by Lemma 
\ref{lem_trace}.} Then
\begin{enumerate}
 \item $C= Rx \oplus (\overline{a}x)^\bot$, and
 \item $x^\bot$ is a direct summand of $C$ which is locally free of rank $7$. 
Furthermore the restriction $q_x$ 
of $q$ to $x^\bot$ is a non-singular quadratic form. 
\end{enumerate}
\end{Lma}

\begin{proof} (1) For each $z\in C$, the element
\[z - q(x)^{-1} b_q(az,x)x\]
is orthogonal to $\overline{a}x$, whence $C=Rx+(\overline{a}x)^\bot$. The sum is moreover
direct since $b_q(x,\overline{a}x)=q(x)\in R^*$, so that $\lambda x \in 
(\overline{a}x)^\bot$ only holds with $\lambda=0$.

\smallskip

\noindent (2) Again we can assume that $R$ is local with maximal ideal 
$\mathfrak{M}$. Since $b_q$ is non-degenerate, there exists $y_K \in 
K=C_{R/\mathfrak{M}}$ such that 
$b_q( x_K, y_K) \not =0 \in K$. We lift $y_K$ to an element 
$y \in C$ and we have $b_q( x,  y) \in R^*$. It follows that 
$C= x^{\bot} \oplus Ry$. 

To show that $q_x$ is a non-singular quadratic form it suffices to consider the 
case when $R=k$ is an algebraically closed field.
We consider the radical 
\[\rad(q_x)= \bigl\{   z \in x^\bot  \, \mid \,  q(z)=b_q(z,x^\bot)=0 \bigr\}.\]
Since $b_q$ is regular, $(x^\bot)^\bot$ is of dimension $1$, whence it is equal 
to $kx$.
But $q(x) \neq 0$, whence  $\rad(q_x)=\{0\}$ in any characteristic, and $q_x$ is 
non-singular. 
\end{proof}

\begin{Rk} If $2 \in R^*$, then $a=\frac{1}{2}$ works, and $C=Rx\oplus 
(\frac{1}{2}x)^\bot=Rx\oplus x^\bot$.
\end{Rk}

We denote by $\mathbb{P}(C^\vee)$ the projective space attached to 
$C^\vee=\mathrm{Hom}_R(C,R)$
using EGA's conventions. For each $R$--ring $S$,  the set 
$\mathbb{P}(C^\vee)(S)$ is that of 
invertible $S$--submodules of $C_S$ which are locally direct summands of $C_S$.
Each point $a$ of $\bS_C(S)$ defines the free $S$--submodule $Sa$ of $C_S$. 
Since $a$ is an invertible octonion, $Sa$ is a direct summand of $C$ 
by Lemma \ref{lem_summand}. We have thus defined an $R$--map $u: \bS_C \to 
\mathbb{P}(C^\vee)$. It is 
$\bm{\mu}_2$--invariant, thus it induces an $R$--map $\underline{u}:\ubS_C   \to 
\mathbb{P}(C^\vee)$.\footnote{The quotient $\ubS_C$ is defined and discussed in 
the appendix.}

\begin{Lma} \label{lem_quadric} The $R$--map $\underline{u} : \underline{\bS}_C   \to 
\mathbb{P}(C^\vee)$ is an open $R$--immersion.
The projective $R$--quadric $Q_C= \{ q_C=0 \} \subset \mathbb{P}(C^\vee)$ is a 
closed complement of the image of
$\underline{u}$.
\end{Lma}

\begin{proof} Let us first show that $\underline{u}$ is a monomorphism.
We are given an $R$--ring $S$ and $\underline{x}, \underline{y} \in 
\underline{\bS}_C(S)$ such that
$\underline{u}(\underline{x})= \underline{u}(\underline{y})$ in 
$\mathbb{P}(C^\vee)(S)$.
Let $S'$ be a flat $R$--cover of $S$ such that $\underline{x}_{S'}$ (resp.\ 
$\underline{y}_{S'}$) lifts to 
an $S'$--point $x' \in \underline{\bS}_C(S')$ (resp.\ $y' \in 
\underline{\bS}_C(S')$.
We then have  $S'x'=S' y'$.
Then $x'= \lambda y'$ for $\lambda \in (S')^*$. By taking the norms we get 
$1= \lambda^2$, i.e.\  $\lambda \in \bm{\mu}_2(S')$. It follows that 
 $\underline{x}_{S'}=  \underline{y}_{S'} \in \underline{\bS}_C(S')$ and thus
 $\underline{x}=  \underline{y} \in \underline{\bS}_C(S)$.
 
The $R$--morphism $\underline{u}$ is thus a monomorphism from a flat 
$R$--scheme 
to a smooth $R$--scheme and such that fibres over $S$ are 
of the same dimension  $7$. According to \cite[$_4$.18.10.5]{EGA4}, $\underline{u}$ 
is an open immersion.

For the last fact it is enough to check that $Q_C(k) \sqcup  
\underline{\bS}_C(k) = \mathbb{P}(C^\vee)(k)$ for each 
$R$--field $k$ which is algebraically closed. In this case 
$\underline{\bS}_C(k)= \bS_C(k)/\bm{\mu}_2(k)$
and it is clear that  $\emptyset = Q_C(k) \cap  \underline{\bS}_C(k)$. 
Further, a point of $\mathbb{P}(C^\vee)(k)$ is a line $L=k x$ in $C \otimes_R 
k$.
If $q_C(x)=0$, it follows that $[kx] \in Q_C(k)$.
If $q_C(x)\not =0$, we may write $q_C(x)=a^{-2}$ with $a \in k^*$, and then
$L= [k (ax)] \in \underline{\bS}_C(k)$.
\end{proof}

\subsection{A result of Knus--Parimala--Sridharan}

As in Remark \ref{Rcomp}, given a composition algebra $C'$ whose norm is 
isometric to $\lambda q_C$,
we can always find an isometry $(C, q_C) \simeq (C',q_{C'})$ mapping $1_C$ to 
$1_{C'}$. 
In the following special case,  Knus, Parimala and Sridharan noticed that 
another kind of isotopes describes 
such objects.

\begin{Prp}\label{prop_KPS} \cite[Remark 4.7]{KPS} Assume that $R$ is a connected ring
and that $\mathrm{Pic}(R)=0$.
If $C$, $C'$ are  octonion $R$--algebras such that there exists an isometry 
$f:(C, q_C) \simlgr (C',q_{C'})$ 
with $f(1_C)=1_{C'}$, then there exists $u \in C^*$ such 
that $f(x) *_{C'} f(y) =  f \Bigl( ( x (y u)) u^{-1} \Bigr)$.
\end{Prp}

According to Remark \ref{rem_KPS}, this algebra  $C'$ is isomorphic to $C^{u^{-1},u} \simeq C^{1,u}$, hence is an isotope.
Their work, as ours, involves the triality phenomenon, and this statement is in some sense the starting point of the present paper. However, as will be seen below, 
our approach is independent of theirs, and our point of view via torsors is of a different flavour. Moreover, our take on triality differs as well. A consequence is the avoidance of discriminant modules, which
lightens the burden of necessary technicalities and enables us to formulate all results over general rings and without reference to the Picard group. 

\section{Triality}\label{sec_triality}
\subsection{Framework and properties}\label{subsec_frame}

Let henceforth $C$ always denote an octonion $R$--algebra. In particular, $C$ is 
endowed with a regular, 
multiplicative quadratic form $q=q_C$ and a natural involution $\kappa:x \mapsto 
\overline x$. We denote by $\bAut(C)$ its automorphism group scheme, which is a 
semisimple
$R$--group scheme of type $\mathrm{G}_2$ and is a closed $R$--subgroup scheme of 
 the linear $R$--group scheme  $\bGL(C)$. Let further $\SO(q_C)$ be the special 
orthogonal group scheme \cite[\S 4.3]{CF} associated to $q_C$. This is a 
semisimple $R$--group scheme of type $\mathrm{D}_4$, and the embedding $\bAut(C) 
\to \bGL(C)$ induces 
a closed immersion $\bAut(C) \to \SO(q_C)$.

We denote by $\RT(C)$ the closed $R$--subscheme of $\SO(q_C)^3$ such that for 
each $R$--algebra
$S$, we have
\[\begin{split}
\RT(C)(S)=\left\{(t_1,t_2,t_3)\in\mathbf{SO}(q)(S)^3 \, | \, 
{t_1}_{T}(xy)=\overline{{t_2}_{T}(\overline{x})}\ 
\overline{{t_3}_{T}(\overline{y})}\right.
\\ \left.\text{for any $S$--ring $T$ and any $x,y\in 
C_T$}\vphantom{\overline{{t_2}_{T}(\overline{x})}}\right\}.   
  \end{split}\]

\begin{Rk} Since the multiplication in $C_T$ is $T$--bilinear, it follows that
\[\RT(C)(S)=\left\{(t_1,t_2,t_3)\in\mathbf{SO}(q)(S)^3 \, | \, \forall x,y\in 
C_S: t_1(xy)=\overline{t_2(\overline{x})}\ 
\overline{t_3(\overline{y})}\right\}.\]
Moreover, if $(t_1,t_2,t_3)\in\mathbf{O}(q)(S)^3$ satisfies the above identity, 
then one can show that $(t_1,t_2,t_3)\in\mathbf{SO}(q)(S)^3$ and thus 
$(t_1,t_2,t_3)\in\RT(q)(S)$. We will use and show this later on.
\end{Rk}

An element of $\RT(C)(S)$ for an $R$--ring $S$ is called a \emph{related 
triple}. In fact, $\RT(C)$ is a closed $R$--subgroup scheme of $\SO(q_C)^3$ 
(i.e.\ related triples can be multiplied component-wise). 

\begin{Rk}\label{rsv}
The definition of related triples here differs from that of \cite{SV} and 
\cite{KPS}, where instead the condition is
\[t_1(xy)=t_2(x)t_3(y).\]
(I.e.\ $(t_1,t_2,t_3)$ is an isotopy from $C_S$ to itself, in the literature 
sometimes called an \emph{autotopy}.) As defined here, a triple is related if 
and only if it satisfies
\[t_1(x\bullet y)=t_2(x)\bullet t_3(y),\]
i.e.\ it is an isotopy with respect to the \emph{para-octonion} multiplication 
$\bullet$ on $C$, defined by $x\bullet y=\overline{x}\ \overline{y}$. 
The (non-unital) composition algebra $(C,\bullet)$ is an instance of what is 
over fields known as a \emph{symmetric composition algebra}. This is the 
approach taken in \cite{KMRT}. Note that if $(t_1,t_2,t_3)$ is related in this sense, then $(t_1,\kappa t_2\kappa,\kappa t_3\kappa)$ is related
in the sense of \cite{SV} and \cite{KPS}. It follows that the two definitions give isomorphic groups.
\end{Rk}

The following lemma, extended from \cite{Eld}, 
shows that the subgroup $A_3$ of $S_3$ acts on $\RT(C)$ by permuting the 
$t_i$'s. The fact that the action of $A_3$ takes this
simple form is the main advantage of using this version of the definition.

\begin{Lma}\label{Lequiv} Let $S$ be an $R$--ring. For 
$(t_1,t_2,t_3)\in\mathbf{SO}(q_C)(S)^3$, the following are equivalent.
\begin{enumerate}
\item $(t_1,t_2,t_3)\in\RT(C)(S)$.
\item $(t_2,t_3,t_1)\in\RT(C)(S)$. 
\item $(t_3,t_1,t_2)\in\RT(C)(S)$.
\item The trilinear form $\Delta_q$ defined on $C_S$ by 
$\Delta_q(w,x,y)=b_q(w,\overline{x}\ \overline{y})$ satisfies 
\[\Delta_q(t_1(w),t_2(x),t_3(y))=\Delta_q(w,x,y)\]
for all $w,x,y\in C_S$.
\end{enumerate}
\end{Lma}

\begin{proof} It is readily verified that 
$\Delta_q(w,x,y)=\Delta_q(x,y,w)=\Delta_q(y,w,x)$. Thus $\Delta_q$ is invariant 
under cyclic permutations, and it suffices to prove the equivalence of the 
first and last statement. Since $t_1$ is orthogonal, the last statement is 
equivalent to
\[\forall w,x,y\in C: b_q(t_1(w),\overline{t_2(x)}\ 
\overline{t_3(y)})=b_q(t_1(w),t_1(\overline{x}\ \overline{y})),\]
which by regularity of $q$ is equivalent to the first statement.
\end{proof}

\begin{Ex}\label{extriple} If $c\in C$ satisfies $q(c)=1$, then the middle Moufang identity $c(xy)c=(cx)(yc)$ implies that $(B_c,R_{\overline{c}},L_{\overline{c}})\in\RT(C)(R)$. By Lemma \ref{Lequiv}, the same holds for $(R_{\overline{c}},L_{\overline{c}},B_c)$ and $(L_{\overline{c}},B_c,R_{\overline{c}})$. These, in a sense, basic triples will be useful in the sequel.
\end{Ex}

The following lemma provides other examples of related triples.

\begin{Lma}\label{laut} Let $S$ be an $R$--ring and $t\in\SO(q_C)(S)$. Then 
$(t,t,t)\in\RT(C)(S)$ if and only if $t\in\bAut(C)(S)$.
\end{Lma}

\begin{proof} If $t\in\bAut(C)(S)$, then $t$ commutes with the involution of 
$C_S$ and satisfies $t(xy)=t(x)t(y)$ for all $x,y\in C_S$. From this it follows 
that $(t,t,t)\in\RT(C)(S)$.

Conversely, assume that $(t,t,t)\in\RT(C)(S)$. To show that $t\in\bAut(C)(S)$ it 
suffices to show that $t$ commutes with the involution. By definition of a 
related triple we have
\[t(\overline{x})=\overline{t(x)}\ \overline{t(1)}=\overline{t(1)}\ 
\overline{t(x)}\]
for all $x\in C_S$. Surjectivity of $t$ implies that $\overline{t(1)}$ belongs 
to the centre $S1$ of $C_S$, whereby $\overline{t(1)}=\alpha1=t(1)$ for some 
$\alpha\in S$. Since
$t$ is an isometry we get $\alpha^2=1$. But then
\[\alpha=t(1)=t(1)t(1)=\alpha^2=1.\]
Thus $t$ fixes $1$ and by the above we get $t(\overline{x})=\overline{t(x)}$ for 
all $x\in C_S$, as required.
\end{proof}

Thus $\bAut(C)$ embeds in $\RT(C)$ by means of the diagonal mapping, and we 
denote by $i: \bAut(C) \to \RT(C)$ that closed embedding, which we will 
henceforth view as an inclusion. We also define $f_i:\RT(C)\to\SO(C)$ as the 
projection
on the $i$th component, for $i=1,2,3$.

\begin{Prp}
 $\bAut(C) \simeq  \RT(C)^{A_3}\simeq \RT(C)^{S_3}$.
\end{Prp}

\begin{proof} By the Lemma \ref{laut}, the embedding $i$ maps $\bAut(C)$ into 
$\RT(C)^{A_3}$, and the projection $f_1$ maps $\RT(C)^{A_3}$ into $\bAut(C)$ and 
satisfies
$f_1\circ i=\Id_{\bAut(C)}$ and $i\circ f_1=\Id_{\RT(C)^{A_3}}$. The first 
isomorphism follows as both these maps are functorial, and the second holds as 
all elements of $S_3$ act trivially on the image of $i$. 
\end{proof}

The centre of $\SO(q_C)^3$ is $\bm{\mu}_2^3$ and its intersection with $\RT(C)$ 
is
\[\bm{\mu}:= \ker( \bm{\mu}_2^3 \xrightarrow{\mathrm{mult}} \bm{\mu}_2).\] 
As the following proposition shows, related triples are determined, modulo 
$\bm{\mu}_2$, by any of their components.

\begin{Lma}\label{lker}
\noindent (1) Let $(t_1,t_2,t_3)\in\RT(C)(R)$ with $t_1=\Id$. Then 
$t_2=t_3=\eta \Id$ for $\eta \in \bm{\mu}_2(R)$.

\smallskip

\noindent (2)  The kernel of $f_1$ is isomorphic to $\bm{\mu}_2$ embedded 
in $\bm{\mu}$ by $\eta \mapsto ( 1, \eta,\eta)$ and similarly for $f_2$ and 
$f_3$.
\end{Lma}

\begin{proof} (1) 
The proof is inspired by \cite{Eld}. Let  $(t_1,t_2,t_3)\in\RT(R)(S)$ with 
$t_1=\Id$. Then
\[\overline{t_2(x)}\ \overline{t_3(y)}=\overline{x}\ \overline{y},\]
or equivalently 
\[t_3(y)t_2(x)=yx\]
for any $x,y\in C_S$, whereby
\[zt_2(x)=t_3^{-1}(z)x\]
for any $x,z\in C_S$. Setting $z=1$ gives $t_2=L_c$ with $c=t_3^{-1}(1)$, and 
setting $x=1$ gives $t_3^{-1}=R_d$ with $d=t_2(1)$. Thus in general
\[z(cx)=(zd)x\]
for any $x,z\in C$. In particular with $x=z=1$ we conclude that $d=c$, and then 
$z(cx)=(zc)x$ for all $x,z\in C$. 
But such an associativity relation only holds if $c=\eta$ for some 
$\eta\in R1$ (Lemma \ref{lem_centre}). Thus $t_2=t_3^{-1}=\eta \Id$. Since e.g.\ 
$t_2$ is an isometry, this gives
\[1=q(t_2(1))=q(\eta 1)=\eta^2,\]
whence the claim.

\smallskip

\noindent (2) The first statement 
 shows that the kernel is contained in $\bm{\mu}_2$.
Conversely, if $t_2=\eta \Id$ and $t_3=\eta \Id$ with $\eta^2=1$, then for all 
$x\in C$,
\[t_1(x)=t_1(x1)=\overline{t_2(\overline{x})}\ \overline{t_3(1)}=\eta^2x1=x,\]
whereby the proof is complete.
\end{proof}

Given this, we observe the following addition to the above discussion on 
automorphisms.

\begin{Prp}\label{pfix} A triple $(t_1,t_2,t_3)\in\RT(C)(S)$ satisfies 
$t_3(1)=t_2(1)\in\bm{\mu}_2(S)1$ if and only if 
$t_1\in\bAut(C_S)$. Moreover, a triple  $(t_1,t_2,t_3)\in\RT(C)(S)$ satisfies 
$t_3(1)=t_2(1)=1$ if and only if $(t_1,t_2,t_3)\in i_S(\bAut(C)(S))$.
\end{Prp}

\begin{proof} Assume that $t_1$ is an automorphism. Then by Lemma \ref{laut}, 
$(t_1,t_1,t_1)$ is a related triple, and 
by Lemma \ref{lker} we have $t_2=t_3=\eta t_1$ for some $\eta\in\bm{\mu}_2(S)$, 
whence
$t_3(1)=t_2(1)\in\bm{\mu}_2(S)1$. Conversely assume that $t_3(1)=t_2(1)=\eta1$ 
with $\eta\in\bm{\mu}_2(S)$. Then for all $x\in C_S$,
\[t_1(x)=t_1(x1)=\overline{t_2(\overline{x})}\ 
\overline{t_3(1)}=\eta\overline{t_2(\overline{x})},\]
whence $t_1=\eta\kappa t_2\kappa$. Similarly one concludes that $t_1=\eta\kappa 
t_3\kappa$. Thus for all $x,y\in C_S$,
\[t_1(x)t_1(y)=\eta^2\kappa t_2\kappa(x)\kappa 
t_3\kappa(y)=\overline{t_2(\overline{x})}\ 
\overline{t_3(\overline{y})}=t_1(xy),\]
Thus $t_1$ is an automorphism and $t_2=t_3=\eta t_1$. This proves the first 
statement and, with $\eta=1$, the
only if-direction of the second. The other direction is clear.
\end{proof}

\begin{Thm}\label{tcover} The $R$--group scheme $\RT(C)$ is semisimple simply 
connected of type $\mathrm{D}_4$. Moreover,
for any $i=1,2,3$, the mapping
$f_i: \RT(C) \to \SO(q_C)$ is the universal cover of $\SO(q_C)$. 
\end{Thm}

The term \emph{universal cover} is understood in the sense of SGA3; see 
\cite[6.5.2]{Co1}.

\begin{proof}
We start with the first statement. Let $C_0$ be the Zorn octonion algebra over 
$\Z$.
 We know that there exists a flat
cover  (that is a finitely presented faithfully flat $R$-ring $S$) 
such that   $C \otimes_R S \simeq C_0 \otimes_\Z S$ \cite[Cor. 4.11]{LPR}.
 It follows that $\RT(C) \times_R S \simeq  \RT(C_0) \times_\Z S$ and 
 the statement boils down to the case of $C_0$ over $\Z$.
 
According to \cite[Prop. 3.6.3]{SV} and Remark \ref{rsv} above,
$\RT(C_0) \times_\Z \Q$
and $\RT(C_0) \times_\Z \F_p$ ($p$ prime) are semisimple simply connected 
algebraic groups of type $\mathrm{D}_4$, in particular 
smooth and connected. Lemma \ref{lem_smooth} then shows that $\RT(C_0)$ is 
smooth over $\Z$.
The affine smooth $\Z$--group scheme $\RT(C_0)$ thus has geometric fibres which 
are
semisimple simply connected, whence by definition, $\RT(C_0)$ is a semisimple 
simply connected
$\Z$--group scheme. 

\smallskip 

For the second statement, it is enough to deal with $f_1$.
According to Lemma \ref{lker}, $\bm{\mu}_2=\ker(f_1)$ and is a central subgroup 
of $\RT(C)$.
We denote by $\underline{\RT}(C)= \RT(C)/ \bm{\mu}_2$ the quotient 
\cite[XXII.4.3.1]{SGA3}; it is a semisimple
$R$--group scheme as well.
Then the induced map $\underline{f_1}: \underline{\RT}(C) \to \SO(q_C)$ has 
trivial kernel.
For each point $s \in \Spec(R)$,  $\underline{f_1}_{s}: \underline{\RT}(C)_{s} 
\to \SO(q_C)_{ s}$ 
is a closed immersion between smooth connected algebraic $\kappa(s)$--groups 
of the same relative dimension $28$. Hence   $\underline{f_1}_{s}$ is an 
isomorphism.
Since $\underline{\RT}(C)$ is flat over $R$, the fibre-wise isomorphism 
criterion \cite[$_4$.17.9.5]{EGA4} enables us to conclude that 
$\underline{f_1}$ is an isomorphism. Then $f_1$ is a central isogeny and since 
$\RT(C)$ is semisimple simply connected, it follows that it
is the simply connected cover of the semisimple $R$--group scheme $\SO(q_C)$.
 \end{proof}

On the other hand, we have the universal $R$--cover $\chi: \Spin(q_C) \to \SO(q_C)$ 
\cite[8.4.0.63]{CF}.
It follows that there is a unique  $R$--isomorphism $F: \Spin(q_C) \simlgr 
\RT(C)$
such  that $f_1 \circ F=\chi$. A construction of this isomorphism will be given 
in the next subsection.
 
\subsection{Explicit construction} \label{subsec_explicit}
For coherence, and in order to relate our work to its predecessors, we will give a constructive proof that the spin group of $(C,q)$ is isomorphic to the group of related triples on $C$. A closely related version can be found in \cite{KPS}. We closely follow the approach of \cite{KMRT} (who work over fields with $2\neq 0$), and
\cite{Eld} (who considers fields of any characteristic). It is independent of the proof given in the previous section. (Conversely, the remainder of the paper is independent 
of this section.)

To this end, we write $\Cl(C,q)$ for the Clifford algebra of the quadratic space 
$(C,q)$, and $\Cl_0(C,q)$ for its even part. Multiplication in $\Cl(C,q)$ will 
always be denoted
by $\cdot$. Consider the $R$--group scheme $\mathbf{Spin}(q)$ defined, for each 
$R$--ring $S$, by\footnote{The condition $u\cdot C_S\cdot u^{-1}=C_S$ implies, 
as an easy calculation shows, that the map
$x\mapsto u\cdot x\cdot u^{-1}$ is an orthogonal operator on $C_S$. In \cite{Kn} 
this latter property is put into the definition.}
\[\mathbf{Spin}(q)(S)=\{u\in\Cl_0(C,q)_S^* \, | \, u\cdot C_S\cdot u^{-1}=C_S 
\text{\ and\ } u\sigma(u)=1\}.\]
We also introduce the $R$--group scheme $\RT'(C)$, defined by
\[\RT'(C)(S)=\left\{(t_1,t_2,t_3)\in\mathbf{O}(q)(S)^3 \, | \, \forall x,y\in 
C_S: t_1(xy)=\overline{t_2(\overline{x})}\ 
\overline{t_3(\overline{y})}\right\}.\]
Similarly to $\RT(C)(S)$, this is a subgroup scheme of $\mathbf{O}(q)^3$. The 
reason for introducing $\RT'(C)$ is merely for technical convenience; we will 
later show that in
fact $\RT'(C)=\RT(C)$.

Recall from \cite[IV.1.2.3]{Kn} that for any $R$--ring $S$, the map 
$\mathrm{incl}\otimes\Id: C\otimes S\to \Cl(C,q)\otimes S$ induces an 
isomorphism of $S$--algebras
\[\Cl(C\otimes S,q\otimes 1)\simeq \Cl(C,q)\otimes S,\]
which obviously preserves the grading and which we shall view as an 
identification.

\begin{Prp}\label{ptrial} Let $S$ be a unital commutative ring and $C$ be an 
octonion algebra over $S$ with quadratic form $q$. The map $\alpha:C\to 
\End_S(C\oplus C)$ defined by
\[x\mapsto\left(\begin{smallmatrix} 
0&L_{\overline{x}}\kappa\\ \kappa L_x&0\end{smallmatrix}\right)\]
induces isomorphisms
\[\alpha':(\Cl(C,q),\sigma)\to(\End_S(C\oplus C),\sigma_{q\bot q})\]
and
\[\alpha'':(\Cl_0(C,q),\sigma)\to(\End_S(C),\sigma_q)\times(\End_S(C),\sigma_q)\]
of algebras with involution. Moreover $\alpha'$ is a morphism of $\mathbb 
Z/2\mathbb Z$-graded algebras with respect to the Clifford and chequerboard 
gradings.
\end{Prp}

Here $\sigma$ is the main involution on $\Cl(C,q)$, $\sigma_{q}$ is the involution on $\End_S(C)$ associated to $b_q$, and 
$\sigma_{q\bot q}$ is the involution on $\End_S(C\oplus C)$ induced by $\sigma_q$.

\begin{Rk}
Note that the entry $L_{\overline{x}}\kappa$ (resp.\ $\kappa L_x$) in the matrix of $\alpha$ is the map of left (resp.\ right) multiplication by $x$ with respect to the 
para-octonion algebra structure on $C$ of Remark \ref{rsv}. While there are other possible ways to construct the map $\alpha$, this choice leads to a cleaner presentation. 
(See also \cite[\S 35]{KMRT}.) We are grateful to A.\ Qu\'eguiner-Mathieu for pointing this out. Note moreover that
\[\alpha(x)=\left(\begin{smallmatrix} 
0&\kappa R_x\\ \kappa L_x&0\end{smallmatrix}\right)=\left(\begin{smallmatrix} 
0&L_{\overline{x}}\kappa\\ R_{\overline{x}}\kappa&0\end{smallmatrix}\right).\]
\end{Rk}

\begin{proof}[Proof of Proposition \ref{ptrial}] The map $\alpha$ is linear and satisfies
\[\alpha(x)^2=\left(\begin{smallmatrix} 
L_{\overline{x}}L_x&0\\0&\kappa L_xL_{\overline{x}}\kappa\end{smallmatrix}\right)=\left(\begin{smallmatrix} 
q(x)\Id_C&0\\0&q(x)\kappa^2\end{smallmatrix}\right)=q(x)\Id_{C\oplus C},\] 
whereby it induces a homomorphism $\alpha'$ as above by the universal property 
of Clifford algebras. The source of $\alpha'$ is, by 
\cite[IV.2.2.3]{Kn} and [IV.1.5.2], an Azumaya algebra of constant rank $2^8$. 
Since $C$ is faithfully projective, the target is an Azumaya algebra of constant 
rank $16^2=2^8$. Thus by 
\cite[III.5.1.18]{Kn}, the map $\alpha'$ is an isomorphism of algebras. It 
respects the grading by construction, from which follows that $\alpha''$ is an 
isomorphism as well. To show that $\alpha'$ and $\alpha''$ are isomorphisms of 
algebras with involution, it suffices
to show that $\alpha'(\sigma(x))=\sigma_{q\bot q}(\alpha'(x))$ for all $x\in C$. 
But then $\sigma(x)=x$, and, setting $A=\alpha'(x)$, we have, for any 
$a,b,c,d\in C$,
\[b_{q\bot q}\left(A\left(\begin{smallmatrix} 
a\\b\end{smallmatrix}\right),\left(\begin{smallmatrix} 
c\\d\end{smallmatrix}\right)\right)=
b_q(\overline{x}\ \overline{b},c)+b_q(\overline{(xa)},d)=
b_q(bx,\overline{c})+b_q(xa,\overline{d})\]
as $\kappa$ is an involutory isometry. By the identities $b_q(zv,w)=b_q(v,\overline{z}w)=b_q(z,w\overline{v})$, this equals
\[b_q(b,\overline{(xc)})+b_q(a,\overline{x}\ \overline{d})=
b_{q\bot q}\left(\left(\begin{smallmatrix} 
a\\b\end{smallmatrix}\right),A\left(\begin{smallmatrix} 
c\\d\end{smallmatrix}\right)\right).\]
Thus $A=\sigma_{q\bot q}(A)$ as required.
\end{proof}

Let now $C$ be an octonion algebra over $R$. The morphism
\[\chi':\Spin(q_C)\to\mathbf{O}(q_C)\]
given, for any $R$--ring $S$ and $u\in\Spin(q_C)(S)$, by $u\mapsto u_1$, where 
for all $x\in C_S$, $u_1:x\mapsto u\cdot x\cdot u^{-1}$, maps $\Spin(q_C)$ 
inside $\SO(q_C)$ \cite[IV.6]{Kn}. We moreover define $u_2,u_3$ by 
\[\alpha'(u)=\left(\begin{smallmatrix} u_3&0\\0&u_2\end{smallmatrix}\right),\] 
noting that they are invertible linear operators on $C_S$ as well. We will write 
$u_i(S)$ for $u_i$, $i=1,2,3$, whenever the ring $S$ needs to be emphasized.

\begin{Thm}\label{Tri} Let $C$ be an octonion algebra over $R$. The map 
\[F:\mathbf{Spin}(q_C)\to\RT'(C),\]
defined, for each $R$--ring $S$ and for each $u\in\mathbf{Spin}(q_C)(S)$, by
\[F_S(u)=(u_1(S),u_2(S),u_3(S)),\]
is an isomorphism of affine group schemes.
\end{Thm}

The following basic lemma, extended from \cite{Eld}, will be useful.

\begin{Lma} Let $(M,q)$ be a regular quadratic module over a ring $S$, and 
assume 
that $f$ is a similitude with multiplier $\lambda\in R$, i.e.\ $q(f(x))=\lambda 
q(x)$ 
for all $x\in M$, and moreover satisfies $b_q(f(x),f(y))=b_q(x,y)$ for all 
$x,y\in M$. Then $\lambda=1$. 
\end{Lma}

\begin{proof} We have
\[b_q(x,y)=b_q(f(x),f(y))=q(f(x)+f(y))-q(f(x))-q(f(y))\]
which, by the hypothesis on $f$, equals
\[\lambda q(x+y)-\lambda q(x)-\lambda q(y)=\lambda b_q(x,y)\]
for all $x,y\in M$. Thus $b_q^\vee(x)=\lambda b_q^\vee(x)\in M^*$ for all $x\in 
M$,
whereby $b_q^\vee=\lambda b_q^\vee$, and, being invertible
\[\Id_{M^*}=\lambda b_q^\vee (b_q^\vee)^{-1}=\lambda \Id,\]
whereby $\lambda=1$.
\end{proof}

\begin{proof}[Proof of Theorem \ref{Tri}] Fix an $R$--ring $S$. We begin by 
showing that the map $F_S$ is well defined. To begin with, 
$u_i(S)\in\mathbf{O}(q)(S)$ for each $i$. For $i=1$, 
this holds by construction. For $i=2,3$, the fact that $\alpha''$ is a morphism 
of 
algebras of involution, together with $\sigma(u)=u^{-1}$, implies that
\begin{equation}\label{bilinear}
b_q(u_i(x),u_i(y))=b_q({\sigma(u_i)}u_i(x),y)=b_q((u_i)^{-1}u_i(x),y)=b_q(x,y). 
\end{equation}
Arguing as in \cite{Eld}, we note that for any $x\in C$,
\[\alpha'(u)\alpha'(x)=\alpha'(u\cdot x\cdot u^{-1}\cdot 
u)=\alpha'(u_1(x))\alpha'(u).\]
Inserting the definition of $\alpha'$ into this equality yields
\begin{equation}\label{insert}
\begin{array}{lll} u_2\kappa L_x=\kappa L_{u_1(x)}u_3&\text{and}& u_3L_{\overline{x}}\kappa=L_{\overline{u_1(x)}}\kappa u_2,
  \end{array} 
\end{equation}
and applying both sides of both equations to $1\in C$ one obtains (substituting 
$x$ for $\overline{x}$)
\[\begin{array}{lll} u_2(x)=\overline{u_3(1)}\ \overline{u_1(\overline{x})}&\text{and}& 
u_3(x)=\overline{u_1(\overline{x})}\ \overline{u_2(1)}.
  \end{array}
\]
Now using the fact that $u_1$ and $y\mapsto \overline{y}$ are isometries of the 
multiplicative quadratic form $q$ on $C$, this gives
\[\begin{array}{lll} q(u_2(x))=q(u_3(1))q(x)&\text{and}& 
q(u_3(x))=q(u_2(1))q(x).
  \end{array}
\]
With \eqref{bilinear} and the above lemma we conclude that $u_2$ and $u_3$ are 
isometries, whence
$F_S(u)\in\mathbf{O}(q)(S)^3$. To show that 
$F_S(u)\in\RT'(C)(S)$, we apply the second equation of \eqref{insert} to an arbitrary $y\in C$. This gives
\[u_3(\overline{x}\ 
\overline{y})=\overline{u_1(x)}\ \overline{u_2(y)},\]
whereby $(u_3,u_1,u_2)$ is a related triple, and, by Lemma 
\ref{Lequiv}, so is $F_S(u)$.
Thus the map $F_S$ is well-defined; it is a group homomorphism since $\chi'$ and 
$\alpha'$ preserve the multiplication, and injective since so is $\alpha'$. 

To show that it is onto, let $(t_1,t_2,t_3)\in\RT'(C)(S)$. Then by Proposition 
\ref{ptrial}, $(t_3,t_2)$ is the image under $\alpha''$ of a 
unique $u\in\Cl_0(q)$. The element $u$ 
satisfies, for all $x\in C_S$,
\[\alpha'(u\cdot x\cdot u^{-1})=\left(\begin{smallmatrix} t_3& 0\\ 
0& 
t_2\end{smallmatrix}\right)\left(\begin{smallmatrix}0&L_{\overline{x}}\kappa\\ \kappa L_x
&0\end{smallmatrix}\right)\left(\begin{smallmatrix} t_3^{-1}& 0\\ 
0& t_2^{-1}\end{smallmatrix}\right)
=\left(\begin{smallmatrix} 0& \kappa t_3L_{\overline{x}}\kappa t_2^{-1}\\ 
t_2\kappa L_x t_3^{-1}\kappa&0\end{smallmatrix}\right).\]
Using the fact that $(t_1,t_2,t_3)$ is related together with Lemma \ref{Lequiv}, 
and the definition of a related triple, the right hand side simplifies to give
\[\alpha'(u\cdot x\cdot 
u^{-1})=\left(\begin{smallmatrix}0&L_{\overline{t_1(x)}}\kappa\\R_{\overline{t_1(x)}}\kappa&0\end{smallmatrix}\right)=\alpha'(t_1(x)).\]
Thus by injectivity, $u\cdot x\cdot u^{-1}=t_1(x)\in C_S$. Moreover, since 
$\alpha''$ is a morphism of algebras with involution, and since $t_2$ 
and $t_3$  are 
isometries, we have
\[\alpha''(\sigma(u))=(\sigma_q(t_3), \sigma_q(t_2))=(t_3^{-1},t_2^{-1}),\]
whereby
\[\alpha''(u\sigma(u))=\alpha''(u)\alpha''(\sigma(u))=(\Id,\Id),\]
whence by injectivity $u\sigma(u)=1$. Thus $u\in\mathbf{Spin}(q)(S)$ and  
$F_S(u)=(t_1,t_2,t_3)$. This proves that $F_S$ is surjective, hence a group 
isomorphism. Functoriality is clear, and the theorem is proved.
\end{proof}

\begin{Cor} \label{cor_RT'} The group schemes $\RT'(C)$ and $\RT(C)$ coincide.
\end{Cor}

\begin{proof} It suffices to show that $\RT'(C)\subseteq \SO(q)^3$, i.e.\ for 
each $R$--ring $S$ and each $(t_1,t_2,t_3)\in\RT'(C)(S)$ we have 
$t_i\in\SO(C)(S)$. 
Now Lemma \ref{Lequiv} holds, with the same proof, if $\RT(C)$ is replaced by 
$\RT'(C)$ and $\SO(q_C)$ by $\mathbf{O}(q_C)$. Thus it suffices to show that 
$t_1\in\mathbf{SO}(q)(S)$. But the theorem provides 
$u\in\mathbf{Spin}(q)(S)$ such that $t_1=u_1$, whereby 
$t_1\in\mathbf{SO}(q)(S)$. 
\end{proof}

\begin{Prp}\label{prop_chi} In the following diagram of affine group schemes, the rows are 
exact, the vertical arrows are natural isomorphisms, and the squares commute. 
\[\xymatrix@1{
\mathbf{1}\ar[r]&\bm{\mu}_2\ar[r]^{\iota}\ar@{=}[d]&\mathbf{Spin}(q) 
\ar[r]^{\chi'}\ar[d]_{F} & \mathbf{SO}(q) \ar@{=}[d]\ar[r] & \mathbf{1}\\
\mathbf{1}\ar[r]&\bm{\mu}_2\ar[r]^{j}&\RT(C) \ar[r]^{f_1}& \mathbf{SO}(q)\ar[r] 
& \mathbf{1}}.
\]
Here, $i$ and $j$ are defined by setting, for each $R$--ring $S$,
\begin{itemize}
\item $\iota_S(\eta)=\eta1$ and
\item $j_S(\eta)=(\Id,\eta\Id,\eta\Id)$,
\end{itemize}
while $F$ and $\chi'$ are defined above.
\end{Prp}

\begin{proof} Let $S$ be an $R$--ring and consider the diagram
\[\xymatrix@1{
\mathbf{1}\ar[r]&\bm{\mu}_2(S)\ar[r]^{\iota}\ar@{=}[d]&\mathbf{Spin}(q)(S) 
\ar[r]^{\chi'}\ar[d]_{F} & \mathbf{SO}(q)(S) \ar[r]^{\mathrm{SN}}\ar@{=}[d] & 
\mathbf{Disc}(S)\ar@{=}[d] \\
\mathbf{1}\ar[r]&\bm{\mu}_2(S)\ar[r]^{j}&\RT(C)(S) \ar[r]^{f_1}& 
\mathbf{SO}(q)(S)\ar[r]^{\mathrm{SN}} & \mathbf{Disc}(S) },
\]
where $SN$ denotes the spinorial norm and $\mathbf{Disc}(S)$ the discriminant 
module (see the appendix). The upper row is exact by \cite[IV.6.2.6]{Kn}. If 
$u\in\mathbf{Spin}(q)(S)$, then by the above theorem
\[f_1\circ F_S(u)=f_1(u_1,u_2 , u_3)=u_1=\chi'(u),\]
whereby the middle square commutes. If $\eta\in\bm{\mu}_2(S)$, then
\[F_S\circ \iota(\eta)=F_S(\eta1)=((\eta1)_1, (\eta1)_2, 
(\eta1)_3)=(\Id,(\eta1)_2, (\eta1)_3).\]
The left square commutes, since
\[((\eta1)_2, (\eta1)_3))=\alpha''(\eta\Id)=(\eta\Id,\eta\Id)=\eta(\Id,\Id),\]
as $\alpha''$ is a morphism of algebras. The statement follows since all maps 
are functoriality is clear, and since $\mathbf{Disc}(S)$ is trivial for 
a faithfully flat $R$--ring $S$.
\end{proof}

\begin{Rk} The map $\chi': \mathbf{Spin}(q) \to \mathbf{SO}(q)$ of Proposition \ref{prop_chi} is a central isogeny.
Since  $\mathbf{Spin}(q)$ is a semisimple simply connected $R$--group scheme, $\chi'$ is then 
a universal cover in the SGA3 sense \cite[6.5.2]{Co1}. Since $\bAut(\bmu_2)=1$,
the map $\chi'$ coincide then with the universal cover $\chi$ considered at the end of  \S \ref{subsec_frame}.
\end{Rk}

\section{Twisting}\label{sec_twist}
\subsection{Torsors}
We consider the usual action of $\SO(q_C)$ on the octonionic unit sphere $\bS_C=\bS_{q_C}$. Recall 
that
\[\bS_C(S)= \{c \in C_S \, | \ q_{C_S}(c)=1\}\]
for any $R$--ring $S$. An important point is  that $\bS_C$ is smooth over $R$ 
(Lemma \ref{lem_sphere}).
We have three actions of $\RT(C)$ on $\bS_C$ and will, throughout,
consider the action  of $\RT(C)$ on $\bS_C \times \bS_C$ defined by
\[{\bf t}.(u,v)= (t_3(u), t_2(v))\]
for any ${\bf t}=(t_1,t_2,t_3)\in\RT(C)(S)$ and $u,v\in\bS_C(S)$.

We establish now in the ring setting a result by Jacobson over fields
of nice characteristic \cite[page 93]{J}, see also \cite[th. 14.69]{Ha} in the real case. Our proof is self-contained. 

\begin{Thm}\label{tstab}
\begin{enumerate}
 \item The stabilizer $\Stab_{\RT(C)}(1, 1)$ is $i(\bAut(C))$.
\item The fppf quotient sheaf $\RT(C)/ \bAut(C)$ is representable by an affine 
$R$--scheme and the induced map $\RT(C)/ \bAut(C) \simlgr \bS_C^2$ is an 
$R$--isomorphism. 
\end{enumerate}
\end{Thm}

\begin{proof}
(1) This follows from Proposition \ref{pfix}.
\smallskip

\noindent (2) Since $\bAut(C)$ is flat over $R$, it follows that the fppf 
quotient $\RT(C)/ \bAut(C)$  is representable
by an  $R$--scheme and the induced map $h: \RT(C)/ \bAut(C) \to \bS_C^2$ is a 
monomorphism \cite[XVI.2.2]{SGA3} .
Furthermore $\RT(C)/ \bAut(C)$ is smooth  since $\RT(C)$ is smooth 
\cite[VI$_B$.9.2.(xii)]{SGA3}.
Since  $\RT(C)/ \bAut(C)$ and $\bS_C^2$ are smooth over $R$ of the same relative 
dimension 14,
 we know that $h$ is an open immersion by \cite[$_4$.18.10.5]{EGA4}.
The proof that $h$ is an isomorphism boils down, by the trick used in the proof 
of Theorem \ref{tcover}, to the 
case of $R = \Z$ and $C=C_0$. 

\medskip

\noindent{\bf Claim.} For each prime $p$, $h_p= h \times_\Z \F_p$ is an 
isomorphism.

\medskip

Since $h_p$ is an open immersion, it is enough to show that $\RT(C_0)(\overline 
\F_p)$ acts transitively on
$\bS_{C_0}^2(\overline \F_p)$ and a fortiori
  to show that for each $q=p^n$, $\RT(C_0)(\F_q)$ acts transitively on
 $\bS_{C_0}^2(\F_q)$.
The coset $\Sigma:=\RT(C_0)(\F_q) / \bAut(C_0)(\F_q)$ is the 
$\RT(C_0)(\F_q)$--orbit of
$(1,1)$ in $\bS_{C_0}^2(\F_q)$. Using \cite[Table page 6]{Hu}, we have
$$
\sharp \Sigma=
\frac{q^{12} (q^2-1) (q^4-1)^2 (q^6-1)}{q^6 (q^2-1) (q^6-1)}= \bigl( q^3 (q^4-1) 
\bigr)^2.
$$
On the other hand, since $C_0$ is hyperbolic of dimension 8, we set $V=\F_q^4$ 
and view the  $\F_q$--points of $\bS_{C_0}$ as the points of $V\oplus V^\vee$ 
having hyperbolic norm
1, i.e.\ the pairs $(v, \phi)$ where $v \in V $ and $\phi \in V^\vee$ satisfy 
$\phi(v)=1$. It follows that $\bS_{C_0}(\F_q)$ has cardinality $\sharp(V 
\setminus \{0\})
\times \sharp(\F_q^3)=(q^4-1)q^3 $. Thus $\sharp \Sigma= \sharp\bigl( 
\bS_{C_0}^2(\F_q) \bigr)$, implying that
 $\Sigma=\bS_{C_0}^2(\F_q)$, and a fortiori that
 $\RT(C_0)(\F_q)$ acts transitively on $\bS_{C_0}^2(\F_q)$.

\smallskip

The Claim is proved. Now $\bS_{C_0}^2$ is a Jacobson scheme, i.e.\ its closed points are dense \cite[$_3$.10.4.6]{EGA4}. Furthermore the 
mapping $\bS_{C_0}^2 \to \Spec(\Z)$ 
maps closed points 
to closed points ({\it ibid}, $_3$.10.4.7). 
The Claim implies that the image of $h$ contains all closed points of 
$\bS_{C_0}^2$; hence $h$ is surjective. 
Thus $h$ is a surjective open immersion and thereby an isomorphism.
\end{proof}

By the above, the map $\Pi:\RT(C)\to\bS_C\times\bS_C$ given, for each $R$--ring 
$S$ and each ${\bf t}\in\RT(C)(S)$, by
\[\Pi_S:{\bf t}=(t_1,t_2,t_3)\mapsto (t_3(1),t_2(1)),\]
is a $\bG$--torsor in the fppf-topology \cite[III.4.1.8]{DG}, where 
$\bG=\bAut(C)$.
Since $\bG$ is smooth over $R$, it is additionally a $\bG$--torsor for the 
\'etale topology \cite[XXIV.8.1]{SGA3}. 
We can actually prove a stronger result.

\begin{Thm}\label{th_zar} (1) There is a natural bijection
$$
 \RT(C)(R)  \backslash \Bigl(\bS_C(R) \times \bS_C(R) \Bigr) \simlgr \ker\Bigl( 
H^1_{\mathrm{Zar}}\bigl(R, \bAut(C)\bigr) \to 
 H^1_{\mathrm{Zar}}\bigl(R, \RT(C)\bigr) \Bigr).
$$

\noindent (2) The $\bAut(C)$--torsor $\Pi:\RT(C)\to\bS_C\times\bS_C$ is trivial 
for the Zariski topology.

\end{Thm}

The statement shows that the defect in transitivity for the  action of  
$\RT(C)(R) $ over 
$\bS_C(R) \times \bS_C(R)$ is encoded in a nice subset of $H^1_{\mathrm{Zar}}(R, 
\bAut(C))$, namely a 
subset of isomorphism classes of Zariski twists of $C$.

\medskip

We write $\mathbf{W}(C)$ for the vector $R$--group scheme of $C$, so that 
\[\mathbf{W}(C)(S)=C_S\ (=C\otimes_RS)\]
for any $R$--ring $S$.
Let $\bE$ be a $\bG= \bAut(C)$--torsor and let  $\bE\wedge^{\bG} \mathbf{W}(C)$ be the 
fppf-twist of $\mathbf{W}(C)$ by $\bE$.
This is the sheaf associated to the presheaf $\bE(C)$ defined, for each 
$R$--ring $S$, by
\[\bE(C)(S)=(\bE(S) \times C_S)/\sim,\]
where 
\[(u,x)\sim(u',x')\Longleftrightarrow \exists g\in \bG(S): (u g, g^{-1} 
x)=(u',x').\]
This is the  twisted octonion 
$R$--algebra  written $^{\bE}C$ for short. We can proceed to the proof of Theorem \ref{th_zar}.

\begin{proof} (1)  According to cohomological properties \cite[prop. 2.4.3]{G3}, 
there is a natural bijection
$$
 \RT(C)(R)  \backslash \Bigl(\bS_C(R) \times \bS_C(R) \Bigr) \simlgr \ker\Bigl( 
H^1_\fppf\bigl(R, \bAut(C)\bigr) \xrightarrow{i^*} 
 H^1_\fppf\bigl(R, \RT(C)\bigr) \Bigr)
$$
 induced  by the characteristic map which to a point $(a,b) \in 
{\bS_C(R) \times \bS_C(R)}$ associates
 the class of the $\bG$--torsor $\bE^{a,b}=\Pi^{-1}(a,b)$. The main point is 
that the orthogonal representation  
 $\bAut(C) \to \mathbf{O}(q_C)$ factors through  $\RT(C)$. 
 It follows that we have the inclusion
 $$
 \ker(i^*)
 \subset 
 \ker\Bigl( H^1_\fppf\bigl(R, \bAut(C)\bigr) \to 
 H^1_\fppf\bigl(R, \mathbf{O}(q_C)\bigr) \Bigr).
 $$
 To refine the statement for the Zariski topology, 
 we consider an $\bAut(C)$--torsor $\bE/R$ such that its class belongs to 
 the second kernel and want to show that it is trivial locally for the Zariski 
topology.
 Let $C'= {^\bE}C$ be the twisted octonion $R$--algebra. 
 The map $H^1_\fppf\bigl(R, \bAut(C)\bigr) \to 
 H^1_\fppf\bigl(R, \mathbf{O}(q_C)\bigr)$ sends $[C']$ to the isometry class of 
the quadratic form $q_{C'}$
 (which is a twisted form of $q_C$). Since  $[C']$ belongs to the second 
kernel, 
 it follows that $q_{C'}$ is isometric to $q_C$. According to  Bix's theorem 
\cite[lemma 1.1]{B}, we have that 
 $C' \otimes_R R_{\mathfrak{P}} \simeq  C \otimes_R R_{\mathfrak{P}}$
 for each prime ideal $\mathfrak{P}$ of $R$. In other words, we have
 $\bE(R_{\mathfrak{P}}) \not=\emptyset$. But $\bE$ is of finite presentation 
over $R$, 
so that 
\[\bE(R_{\mathfrak{P}}) = \varinjlim_{f \not \in \mathfrak{P}} \bE(R_f).\]
 For each prime ideal $\mathfrak{P}$ there thus exists  an element 
$f_{\mathfrak{P}} \not \in \mathfrak{P}$ such that 
 $\bE(R_{f_{\mathfrak{P}}}) \not = \emptyset$. The  $f_{\mathfrak{P}}$'s 
generate $R$ where  $\mathfrak{P}$ runs over 
 the maximal ideals of $R$.  In particular there exists
 a partition of unity $1_R= f_1 +\dots +f_r$ such that $\bE(R_{f_i})\not = 
\emptyset$ for $i=1,...,n$.
 Thus $\bE$ is a $\bG$--torsor which is Zariski locally trivial.
 We have shown (1).
 
 \smallskip
 
 \noindent (2) We apply (1) to the coordinate ring $R[\bS_C^2]$ and to the 
universal point \linebreak
 $\eta \in  \bS_C^2\bigl(R[\bS_C^2]\bigr)$. Its image by the characteristic map 
is nothing but the class of the 
 $\bG$--torsor $\Pi: \RT(C) \to \bS_C^2$. Then $(1)$ yields that $\Pi$ admits 
sections locally with respect 
 to the Zariski topology.
\end{proof}

A natural question to ask at this stage is whether this construction is trivial 
or not.
We provide here a variant of \cite{G2}.

\begin{Ex}\label{ex_mimura} Assume that $R= \mathbb{R}$ is the field of real numbers and that $C$ 
is the division algebra of Cayley--Graves octonions.
In this case, $\bG=\bAut(C)$ is the real semisimple anisotropic algebraic group 
of type $G_2$ and 
$S_C$ is the standard real sphere with equation $\sum_{i=1}^8 \, x_i^2=1$. 
We claim that the $\bG$--torsor $\Pi: \RT(C) \to \bS_C^2$ is non-trivial. 
If it is trivial, there is a $\bG$--isomorphism  $\RT(C) \simlgr  {\bS_C^2 
\times_{\mathbb{R}} \bG}$ and in particular 
the morphism $\bG \to \RT(C) \simeq \Spin_8$ admits an algebraic retraction.
When taking the real points, we have that the inclusion of topological groups 
$\bG(\mathbb{R}) \to \Spin_8(\mathbb{R})$ admits a 
continuous retraction. Therefore for each $j \geq 1$, the homotopy group 
$\pi_j(\bG(\mathbb{R}),1)$ 
is a direct summand of $ \pi_j( \Spin_8(\mathbb{R}))$. According to the tables  
\cite[p. 970]{M},
we have $\pi_6(\bG(\mathbb{R}),1)= \Z/3\Z$
and $\pi_6(\Spin_8(\mathbb{R}),1)= 0$ whence a contradiction. Thus the 
$\bG$--torsor $\Pi: \RT(C) \to \bS_C^2$ is non-trivial.
Note also that the same fibration extended to the field $\mathbb{C}$ of complex 
numbers is still non-trivial. Proposition \ref{prop_marlin} provides an algebraic proof of this fact.
\end{Ex}

\subsection{Isotopes as twisted algebras}

We shall now see how the isotopes of Section \ref{sec_background} appear 
naturally as the twists above. 
We are given $a,b\in \bS_C(R)$ and consider the $\bG$--torsor
$\bE^{a,b}= \Pi^{-1}(a,b)$. It gives rise to the sheaf of twisted octonion 
$R$--algebras
$\bE^{a,b}(C)$. This is the sheaf associated to the presheaf $\bE^{a,b}(C)$ 
defined, for each $R$--ring $S$, by
\[\bE^{a,b}(C)(S)=(\Pi_S^{-1}(a_S,b_S)\times C_S)/\sim,\]
where 
\[({\bf t},x)\sim({\bf t}',x')\Longleftrightarrow \exists g\in \bG(S): ({\bf 
t}g,x)=({\bf t}',g(x')).\]

\begin{Rk}\label{Raut} Note that if $\Pi_S({\bf t})=\Pi_S({\bf t}')$, then ${\bf 
t}^{-1}{\bf t}'\in i_S(\bG(S))$.
 \end{Rk}
 
The linear structure on $\bE^{a,b}(C)(S)$ is induced by (cf.\ 
\cite[III.4.3.3]{DG})
\[({\bf t},x)+({\bf t}',x')=({\bf t},x+t_1^{-1}t_1'(x'))\]
and, for $s\in S$,
\[s({\bf t},x)=({\bf t},sx),\]
and the multiplication by 
\[({\bf t},x)\cdot({\bf t}',x')=({\bf t},xt_1^{-1}t_1'(x')).\]
The unity is the class of $({\bf t},1)$, which is independent of the choice of ${\bf t}\in \Pi_S^{-1}(a_S,b_S)$ by Remark \ref{Raut}. Using this and the fact that $q(c)1=c\overline{c}$ for any element $c$ of an octonion algebra, one finds that the quadratic form of $\bE^{a,b}(C)(S)$ coincides with that of $C_S$.

The following observation will be useful.

\begin{Prp}\label{Ptrialiso}  Let $a,b\in \bS_C(R)$ and 
$(t_1,t_2,t_3)\in\RT(C)(S)$ for an $R$--ring $S$. Then $t_1$ is an isomorphism 
$C_S\to C_S^{a,b}$
if and only if $t_2(1)=\eta b_S$ and $t_3(1)=\eta a_S$ for some 
$\eta\in\bm{\mu}_2(S)$.
\end{Prp}

Abusing notation, we write $a$ and $b$ for $a_S$ and $b_S$, respectively.

\begin{proof} If $t_2(1)=b$ and $t_2(1)=a$, then
\[t_1(x)=t_1(1x)=\overline{t_2(1)} \ \overline{t_3(\overline{x})}=\overline{b} \ 
\overline{t_3(\overline{x})},\]
whereby $L_{\overline{b}}\kappa t_3\kappa=t_1$, and thus $\kappa 
t_3\kappa=L_bt_1$. Likewise 
$\kappa t_2\kappa=R_at_1$. This implies, by the definition of a related triple, 
that for all $x,y\in C_S$,
\[t_1(xy)=\kappa t_2\kappa (x)\kappa t_3\kappa 
(y)=R_at_1(x)L_bt_1(y)=t_1(x)*_{a,b}t_1(y).\]
Since $C^{\eta a,\eta b}=C^{a,b}$ for any $\eta\in\bm{\mu}_2(S)$, this concludes 
the if-part.

Conversely, if $t_1:C_S\to C_S^{a,b}$ is an isomorphism, then $(t_1,\kappa R_a 
t_1\kappa,\kappa L_b t_1\kappa)$ is a related triple. Since $t_1(1)=(ab)^{-1}$ 
we get
$\kappa R_a t_1\kappa(1)=b$ and $\kappa L_b t_1\kappa(1)=a$. Lemma \ref{lker} 
then implies that $t_2(1)=\eta b$ and $t_3(1)=\eta a$ for some 
$\eta\in\bm{\mu}_2(S)$.
\end{proof} 

We can now establish the link between twisted forms and isotopes, and formulate the main theorem of this section.

\begin{Thm}\label{thm_main} Let $a,b\in\bS_C(R)$. Consider, for each $R$--ring 
$S$, the map
\[\begin{array}{ll}\Theta_S^{a,b}:\Pi_S^{-1}(a_S,b_S)\times C_S\to 
C_S^{a,b},&({\bf t},x)\mapsto t_1(x).\end{array}\]
Then these maps induce a natural isomorphism of fppf-sheaves of algebras
\[\widehat\Theta^{a,b}:\bE^{a,b}\wedge^\bG \mathbf{W}(C)\xrightarrow{\sim} \mathbf{W}(C^{a,b}).\]
\end{Thm}

Naturality means that for any morphism $\phi:S\to T$ of $R$--rings, the diagram
\[\xymatrixcolsep{3pc}\xymatrix@1{
&\bE^{a,b}(C)(S)\ar[r]^{\widehat\Theta_S^{a,b}}\ar[d]^{\widehat\phi}& C_S^{a,b} 
\ar[d]^{\widehat\phi} \\
&\bE^{a,b}(C)(T)\ar[r]^{\widehat\Theta_T^{a,b}}&C_T^{a,b}}
\]
commutes. Note that the maps are well-defined by Remark \ref{Rbase}.
 
\begin{proof} By the universality of the sheaf associated to a presheaf, it 
suffices to prove that for each $S$ such that 
$\Pi_S^{-1}(a_S,b_S)\neq\emptyset$, the map
$\widehat\Theta_S^{a,b}$ is a natural isomorphism of algebras. Consider such an 
$S$.

Firstly, the map $\widehat\Theta_S^{a,b}$ is well defined, since if $({\bf 
t},x)\sim({\bf t}',x')\in \Pi_S^{-1}(a_S,b_S)\times C_S$, then ${\bf t}'={\bf 
t}g$ and $x'=g^{-1}(x)$ for some $g\in \bG(S)$, whereby 
\[t_1'(x')=t_1gg^{-1}(x)=t_1(x).\]

Secondly, $\widehat\Theta_S^{a,b}$ is an algebra homomorphism. Indeed, it is 
linear, since
\[\Theta_S^{a,b}(s({\bf t},x))=\Theta_S^{a,b}({\bf 
t},sx)=t_1(sx)=st_1(x)=s\Theta_S^{a,b}(({\bf t},x))\]
for all $s\in S$, and
\[\Theta_S^{a,b}(({\bf t},x)+({\bf t}',x'))=\Theta_S^{a,b}({\bf 
t},x+t_1^{-1}t_1'(x'))=t_1(x+t_1^{-1}t_1'(x'))=t_1(x)+t_1'(x'),\]
which by definition is $\Theta_S^{a,b}(({\bf t},x))+\Theta_S^{a,b}(({\bf 
t}',x'))$. With respect to the multiplication,
\[\Theta_S^{a,b}(({\bf t},x)({\bf t}',x'))=\Theta_S^{a,b}({\bf 
t},xt_1^{-1}t_1'(x'))=t_1(xt_1^{-1}t_1'(y)),\]
and since ${\bf t}\in\Pi_S^{-1}(a,b)$, by Proposition \ref{Ptrialiso} this 
equals
\[t_1(x)*_{a,b}t_1t_1^{-1}t_1'(x')=t_1(x)*_{a,b}t_1'(y)=\Theta_S^{a,b}({\bf 
t},x)*_{a,b}\Theta_S^{a,b}({\bf t}',x').\]

Thirdly, $\widehat\Theta_S^{a,b}$ is injective, since if $\Theta_S^{a,b}(({\bf 
t},x))=t_1(x)=0$, then $x=0$ since $t_1$ is invertible, whereby the class of 
$({\bf t},x)$ is the zero element of $\bE^{a,b}(C)(S)$;
it is surjective since given $x\in C_S^{a,b}$, any ${\bf 
t}\in\Pi_S^{-1}(a,b)\neq\emptyset$ satisfies $\Theta_S^{a,b}({\bf 
t},t_1^{-1}(x))=x$. Thus $\widehat\Theta_S^{a,b}$ is an isomorphism. It is 
natural since the commutativity of the square above is the statement that, for 
each $(t_1,t_2,t_3)\in\RT(C)(S)$, we have the equality 
${t_{1}}_T\widehat\phi=\widehat\phi t_1$, which follows from the 
definition of ${t_{1}}_T$.
\end{proof}
 
In view of Remark \ref{Rreduce} and the discussion preceding it, this implies 
that any twisted form of the octonion algebra $C$, where the twist is effected by the torsor above, is an 
isotope. 
By that remark and the proposition preceding it, one may, without loss of 
generality (up to isomorphism),
assume that $b=1$ or, equivalently, $b=a^{-1}$. This and other variants are 
discussed in Section \ref{sec_variants}, where we will also show that the twists 
above account for all octonion algebras isometric to $C$.

\section{Related constructions}
\subsection{Twisted automorphism groups}
We will briefly discuss the automorphism groups of the twisted algebras. We set 
$\bG^{a,b}= \bAut(C^{a,b})$ for $a, b \in \bS_C(R)$. 
It is the twisted group scheme of $\bG$ by the torsor $\bE^{a,b}$.
According to Demarche's compatibility \cite[\S 2.4.1]{G3}, we have a natural 
identification
$$
\bG^{a,b} \simlgr \Bigl\{ t \in \RT(C)  \, \mid \, t.(a,b)=(a,b) \Bigr\}.
$$
In other words, $\bG^{a,b}$ is nothing but the stabilizer of $(a,b)$ for the 
action of
$\RT(C)$ on $\bS_C^2$. From another point of view, it consists of the fixed 
points of a twisted $A_3$-action on $\RT(C)$, which is given explicitly in the 
following proposition.

\begin{Prp} Let $a,b\in\bS_C(R)$. 
\begin{enumerate}
\item The map
\[T^{a,b}:\RT(C^{a,b})\to\RT(C)\]
defined, for each $R$--ring $S$ and each $(t_1,t_2,t_3)\in\RT(C^{a,b})(S)$, by
\[(t_1,t_2,t_3)\mapsto (t_1,B_bR_at_2R_{\overline{a}}B_{\overline{b}},B_aL_bt_3L_{\overline{b}}B_{\overline{a}}),\]
is an isomorphism of $R$--group schemes.
\item The transfer under $T^{a,b}$ of the automorphism $(t_1,t_2,t_3)\mapsto(t_2,t_3,t_1)$ of $\RT(C^{a,b})(S)$ is the automorphism
\[\widehat\sigma^{a,b}: \RT(C)\to \RT(C)\]
of order three defined on $\RT(C)(S)$ for each $R$--ring $S$ by
\[(t_1,t_2,t_3)\mapsto (R_{\overline{a}}B_{\overline{b}},L_{\overline{a}}R_b,B_aL_b)(t_2,t_3,t_1)(R_{\overline{a}}B_{\overline{b}},L_{\overline{a}}R_b,B_aL_b)^{-1}.\]
\item The group $\bG^{a,b}$ is isomorphic to the fixed locus of $\widehat\sigma^{a,b}$. 
\end{enumerate}
\end{Prp}

In other words, $\widehat\sigma^{a,b}=\mathrm{Int}(\mathbf{s}_{a,b})\widehat\sigma$ with $\mathbf{s}_{a,b}=(R_{\overline{a}}B_{\overline{b}},L_{\overline{a}}R_b,B_aL_b)\in\RT(C)(R)$, 
and $\widehat\sigma$ being the standard action of $A_3$, induced by the element $\sigma=(123)$. Note further that the map $T^{a,b}$ is, by construction, compatible with the covering maps from $\RT(C^{a,b})$ and $\RT(C)$ to $\SO(q_{C^{a,b}})=\SO(q_C)$ given by the 
projection on the first component of a related triple.

\begin{proof} The proof of (1) is a straight-forward verification; as is the proof that (2) follows from (1). For (3) we know that $\bG^{a,b}$
embeds in $\RT(C^{a,b})$ as the fixed locus of the automorphism in (2), the image of which in $\RT(C)$ is the fixed locus of the transfer $\widehat\sigma^{a,b}$.
\end{proof}

\subsection{Trialitarian action on the torsor}\label{subsec_rel_trial}
Since $A_3$ acts trivially on $\bG=\bAut(C)$, the trialitarian action of $A_3$ 
on 
$\RT(C)$ induces an action of $A_3$ on the $\bG$--torsor \break $\Pi: \RT(C) \to 
\bS_C^2$.

\begin{Lma} The action of the cycle $\sigma= (123)$ on  $\bS_C^2(R)$ maps
$(x,y)$ to $( \overline{y} \,  \overline{x}, x)$.
\end{Lma}

\begin{proof} We write $(x',y')=\sigma(x,y)$.
We may assume that $R$ is local.  Then there exists $\mathbf{t}= (t_1,t_2,t_3) 
\in \RT(C)(R)$ such that
$t_3(1)=x$ and $t_2(1)=y$.  Since $\sigma({\bf t})= (t_2,t_3,t_1)$, we have 
$x'=t_1(1)$ and
$y'= t_3(1)=x$. We finish by using the relation
$x'=t_1(1)=\overline{t_2(\overline{1})}\,  \overline{t_3(\overline{1})}= 
\overline{y} \, \overline{x}$.
\end{proof}

This gives rise to the following formulae for isotopes.

\begin{Cor} Let $a,b\in\bS_C(R)$. 
Then the $R$--algebras $C^{a,b}$, $C^{b^{-1}a^{-1} , a}$ and
$C^{b, b^{-1}a^{-1}}$ are isomorphic.
\end{Cor}

It is reassuring that these relations can be derived from the known relations of 
Section \ref{subsec_relation}.
Indeed, McCrimmon's relation establishes an isomorphism $C^{b^{-1}a^{-1}, a} \simeq 
C^{a b^{-1},1}$
which is isomorphic to $C^{a,b}$ by Proposition \ref{prop_formulae}.(4).

\begin{Rk}\label{rk_rel} The above correspondence between the orbits of the action of $\RT(C)$ on $\bS_C\times\bS_C$ and the isomorphism
classes of the isotopes implies 
some of the other isomorphism relations between these. Consider, for instance, 
the relation $C^{a,a}\simeq C$. By Example \ref{extriple}, $\mathbf t= (B_{\overline{a}}, R_a, L_a)$ belongs to $\RT(C)(R)$. Then $\Pi(\mathbf t)=(a,a)$, whence $C^{a,a} 
\simeq C$. We will exploit this point of view further in Section 
\ref{sec_particular}.
\end{Rk}

\subsection{Compositions of quadratic forms}
We remind the reader of the notion of a composition of quadratic forms, which is 
essentially a generalization of that of a composition algebra.
Let $(M_1,q_1)$, $(M_2,q_2)$, $(M_3,q_3)$ be non-singular quadratic
$R$--forms of
common rank. A \emph{composition of quadratic forms} is a bilinear $R$--map $f: 
M_2 \times M_3 \to M_1$ such that
$q_1( f(m_2,m_3))= q_2(m_2) \, q_3(m_3)$ for each $m_2 \in M_2$, $m_3\in
M_3$. More precisely, it is the datum
$\calM= (M_1,q_1,M_2,q_2,M_3,q_3,f)$ for which the notion of isomorphism
is clear.

Our favourite example is the composition $\calM_C$  attached to the
octonion algebra $C$, i.e.\  $M_1=M_2=M_3=C$ and
$f$ is the octonionic multiplication.
In this case we have an isomorphism $\RT'(C)(S) \simlgr \bAut(\calM_C)(S)$ by mapping
a triple $(t_1, t_2,t_3)$ to $(t_1,\kappa t_2\kappa, \kappa t_3\kappa )$ where $\kappa$
stands for the octonionic involution on $C$.
It follows that the $R$--functor $\bAut(\calM_C)$ is representable by the 
$R$--group scheme $\RT'(C)$ which is nothing but $\RT(C)$ according to Corollary \ref{cor_RT'}.

It follows that the set $H^1_\fppf(R, \RT(C))$ classifies the compositions
of rank eight which are locally isomorphic for the flat topology 
(or even \'etale topology) to $\calM_C$.
The centre $\ker( \bmu_2^3 \to \bmu_2)$ of $\RT(C)$ gives an action of
the commutative group
$H^1_\fppf(R,  \ker( \bmu_2^3 \to \bmu_2))$ on $H^1_\fppf(R, \RT(C))$.
Now $H^1_\fppf(R,\bmu_2)\simeq\mathbf{Disc}(R)$ (see the appendix), and 
since the exact sequence $1 \to \bmu \to \bmu_2^3 \to \bmu_2 \to 1$ is split we have
\begin{equation}\label{eq_disc}
H^1_\fppf(R,  \ker( \bmu_2^3 \to \bmu_2))= \ker\bigl( \mathbf{Disc}(R)^3
\to \mathbf{Disc}(R) \bigr).
\end{equation}
We will use this to explain how that action  can be understood at the level of
$R$--forms of $\calM_C$.  Given non-singular quadratic forms  $(\calL_1, \theta_1),
(\calL_2, \theta_2), (\calL_3, \theta_3)$ of rank one and an isometry
$\epsilon : (\calL_2 \otimes_R \calL_3 ,
\theta_2 \otimes \theta_3) \simlgr   (\calL_1, \theta_1)$, we can modify an
$R$--composition
\[\calM=(M_1,q_1,M_2,q_2,M_3,q_3,f)\]
as
$$
\bigl(\calL_1 \otimes_R M_1,\theta_1 \otimes q_1, \calL_2 \otimes_R M_2,
\theta_2 \otimes q_2 , \calL_3 \otimes_R M_3,
\theta_3 \otimes q_3, \epsilon \otimes f \bigr).
$$
By \eqref{eq_disc}, this gives the action of
$H^1_\fppf\bigl(R,  \ker( \bmu_2^3 \to \bmu_2)\bigr)$ on $H^1_\fppf(R, \RT(C))$.

\begin{Rk} Let $C'$ be an $R$--form of $C$ such that the compositions
$\calM_C$ and $\calM_{C'}$ are isomorphic.
Then there exists a triple $(t_1,t_2,t_3)$ of isometries between $(C,q)$
and $(C',q')$ such that
$$
t_1( x *_C y) = t_2(x) *{_ {C'}} t_3(y)
$$
for all $x,y \in C$. In view of the introductory part of Section 
\ref{sec_background}, this implies that the kernel of $H^1_\fppf(R,\bG) \to
H^1_\fppf(R,\RT(C))$ consists of isotopes  of $C$. This confirms a consequence 
of the more precise construction given in Theorem \ref{thm_main}.
\end{Rk}

\section{Variants}\label{sec_variants}

One can modify the torsor defined above in several ways in order to describe octonion algebras under various similarity relations. The aim
of this section is to study these different variants.

\subsection{Codiagonal variant}

We consider the $\bAut(C)$--torsor $\Pi: \RT(C) \to \bS_C^2$ and the codiagonal 
map
$\nabla : \bS_C \to \bS_C^2$, $a \mapsto \nabla(a)= (a,a^{-1})$.
We denote by 
$\RT(C)^{\nabla} \to \bS_C$ the 
pull-back of the $\bAut(C)$--torsor $\Pi$ with respect to $\nabla$.
We observe that $\RT(C)^{\nabla}= \RT(C) \times_{\bS_C^2} \bS_C$ is a closed 
$R$--subscheme of $\RT(C)$. 
For each $R$--ring $S$, we have
$$
\RT(C)^{\nabla}(S)= \bigl\{ (t_1,t_2,t_3) \in  \RT(C)(S) \, \mid \,  
t_3(1)t_2(1)=1 \bigr\}
$$
$$
\qquad \qquad = \bigl\{ (t_1,t_2,t_3) \in  \RT(C)(S) \, \mid \,  t_1(1)=1 
\bigr\}.
$$
In terms of the covering  $f_1:  \RT(C) \to \SO(q_C)$, $\RT(C)^{\nabla}$ is the 
inverse image of 
the $R$--subgroup scheme  $\bigl\{ t_1 \in   \mathbf{SO}(q_C) , \mid \, t_1(1) 
=1\bigr\}$.
We put $C_1=1^{\perp_q} \subset C$. By Lemma \ref{lem_summand}, $C_1$ is a 
locally free $R$--submodule of rank $7$ of $C$
and the restriction $q_{C_1}$ of $q_C$ to $C_1$ is non-singular. 
We need to be careful  when dealing with the orthogonal  $R$--group scheme
$\mathbf{O}(q_{C_1})$; it fits in an exact sequence \cite[4.3.0.24]{CF}
$$
1 \to \mathbf{SO}(q_{C_1}) \to \mathbf{O}(q_{C_1}) \xrightarrow{det} \bmu_2 \to 1
$$
where  the special orthogonal $R$--group scheme $\mathbf{SO}(q_{C_1})$ is semisimple of type 
$B_3$.

\begin{Lma}\label{lem_nabla}

\begin{enumerate}
 \item The restriction map $\bigl\{ t_1 \in   \mathbf{O}(q_C)\mid \, t_1(1) =1\bigr\} \to \mathbf{O}(q_{C_1})$,
 $f \mapsto f_{\mid C_1}$  induces  an isomorphism  of group schemes   
 $$
 \mathbf{SO}(q_C)_1:=\bigl\{ t_1 \in   \mathbf{SO}(q_C) \mid \, t_1(1) 
=1\bigr\} \simlgr  \mathbf{SO}(q_{C_1}).$$
\item There is a unique $R$--homomorphism  $h: \Spin(q_{C_1}) \to \RT(C)$
such that the following exact diagram commutes 
\[\xymatrix@1{
\mathbf{1} \ar[r] &\bm{\mu}_2 \ar[r]^{} \ar[d]^{\simeq} & \mathbf{Spin}(q_{C_1})
 \ar[r] \ar[d]_{h} & \mathbf{SO}(q_{C_1}) \ar[r] \ar[d]  & \mathbf{1}  \\
 \mathbf{1} \ar[r]& \bm{\mu}_2 \ar[r] & \RT(C) \ar[r]^{f_1} & \mathbf{SO}(q_C) 
\ar[r] & \mathbf{1}.
}\]
Furthermore $h$ induces an $R$--isomorphism $\Spin(q_{C_1}) \simlgr 
\RT(C)^\nabla$.
\end{enumerate}
\end{Lma}

\begin{proof} (1) 
In the case of fields, the statement is Proposition 2.2.2 of \cite{SV} but some precaution
has to be taken in characteristic 2 since Springer--Veldkamp work with
the classical theory of algebraic groups which identifies $\mathbf{SO}(q_{C_1})$ and
$\mathbf{O}(q_{C_1})$. We treat that at first.

\smallskip

\noindent{\it Case of a field $k$ of characteristic $2$.} Without loss of generality, 
we may assume that $k$ is algebraically closed. The proof of
the quoted result shows that the map $v: \bigl\{ t_1 \in   \mathbf{O}(q_C)\mid \, t_1(1) =1\bigr\} \to \mathbf{O}(q_{C_1})$
  induces  an isomorphism  of abstract groups
 $\mathbf{SO}(q_C)_1(k) \simlgr  \mathbf{SO}(q_{C_1})(k)$. 
 We denote by $J= (\mathbf{SO}(q_C)_1)_{\mathrm{red}}$ the largest smooth subgroup of $\mathbf{SO}(q_C)_1$. We then have $J(k)= \mathbf{SO}(q_C)_1(k)$, and
 the  map $J \to \mathbf{O}(q_{C_1})$ factors trough a homomorphism   $v': J \to \mathbf{SO}(q_{C_1})$ such that the induced map
 $J(k)  \to \mathbf{SO}(q_{C_1})(k)$ is bijective. It follows that $v'$ is an isogeny 
 so that $J$ and $\mathbf{SO}(q_C)_1$  have the  same dimension as  $\mathbf{SO}(q_{C_1})$, i.e.\ $21$.
The same reference shows that the map $\Lie(v'): \Lie(\mathbf{SO}(q_C)_1) \to \Lie(\mathbf{SO}(q_{C_1}))$
is an isomorphism, so that $\Lie(\mathbf{SO}(q_C)_1)$ has dimension 21, whence $\mathbf{SO}(q_C)_1$ is smooth \cite[II.5.2.1]{DG}.
It follows that $J= \mathbf{SO}(q_C)_1$ and that $v': \mathbf{SO}(q_C)_1 \to \mathbf{SO}(q_{C_1})$ is an isogeny
such that $\Lie(v')$ is an isomorphism. Thus $v'$ is \'etale and the bijectivity of $v(k)$ implies that $v'$ is an 
isomorphism as desired.

\smallskip

\noindent{\it General case.}
By faithfully flat descent, this boils down as usual to the split case over $\Z$.
Using the result over fields, the $\Z$-group scheme $\mathbf{SO}(q_C)_1$ has smooth connected fibres of
common dimension 21. Lemma \ref{lem_smooth} shows that $\mathbf{SO}(q_C)_1$ is smooth, hence
flat. It follows that the schematic image of $\mathbf{SO}(q_{C_1}) \to \mathbf{O}(q_{C_1})$
is the schematic closure of $v\bigl(  \mathbf{SO}(q_C)_1 \times_\Z \Q \bigr)=  \mathbf{SO}(q_{C_1}) \times_\Z \Q$,
hence is  $\mathbf{SO}(q_{C_1})$.  The map  $v$ factors trough $\mathbf{SO}(q_{C_1})$
and we consider then $v':   \mathbf{SO}(q_C)_1  \to \mathbf{SO}(q_{C_1})$.
We conclude with the fibre-wise isomorphism criterion 
\cite[$_4$.17.9.5]{EGA4} that $v'$ is an isomorphism.

\smallskip

\noindent (2) We pull back the extension $\mathbf{1} \to \bm{\mu}_2 \to \RT(C) 
\xrightarrow{f_1} \mathbf{SO}(q_C)  \to \mathbf{1}$ (see Theorem \ref{tcover})
by the composite map $\mathbf{Spin}(q_{C_1}) \to \mathbf{SO}(q_{C_1})  \to 
\mathbf{SO}(q_C)$ and get a
central extension of
$R$--group schemes $\mathbf{1} \to \bm{\mu}_2 \to \bH  \to 
\mathbf{Spin}(q_{C_1}) \to \mathbf{1}$.
We observe that $\bH$ is a closed subgroup scheme of $\RT(C)$.
Since $\mathbf{Spin}(q_{C_1})$ is simply connected, this sequence is (uniquely) 
split \cite[6.5.2.(iii)]{Co1}, i.e. 
there exists a unique $R$--homomorphism  $h_0: \mathbf{Spin}(q_{C_1}) \to \bH$ 
splitting the above sequence. By construction 
the following diagram
\[\xymatrix@1{
\mathbf{Spin}(q_{C_1})
 \ar[r] \ar[d]_{h} & \mathbf{SO}(q_{C_1}) \ar[r] \ar[d]  & \mathbf{1}  \\
\RT(C) \ar[r]^{f_1} & \mathbf{SO}(q_C) \ar[r] & \mathbf{1}
}\]
commutes. As explained before, the homomorphism 
$h$ satisfying that commutativity is already unique.  
It remains to be shown that the induced map $h^\sharp: \bm{\mu}_2  \to 
\bm{\mu}_2$ is the identity.
Over a field of odd characteristic, this follows of \cite[Example 17.1]{Ga}.
To handle the general case, one  can assume again that $R=  \mathbb{Z}$ and that 
$C$ is the split octonion algebra.
Then  $h^\sharp$ is the identity or the trivial map. It cannot be trivial since 
it is not over $\mathbb{F}_3$.
We conclude that $h^\sharp$ is the identity as desired. 
It follows that the diagram 
\[\xymatrix@1{
 \mathbf{Spin}(q_{C_1})
 \ar[r] \ar[d]_{h} & \mathbf{SO}(q_{C_1}) \ar[d]  \\
  \RT(C) \ar[r]^{f_1} & \mathbf{SO}(q_C)  
}\]
is Cartesian so that $\mathbf{Spin}(q_{C_1})$ is $R$--isomorphic to 
$\RT(C)^\nabla$.
\end{proof}

\medskip

Since $\bAut(C)$ is an  $R$--subgroup scheme of $\RT(C)^\nabla$, we obtain an 
embedding
$\bAut(C) \to \Spin(q_{C_1})$. Furthermore we have a Cartesian diagram
\[\xymatrix@1{
\mathbf{Spin}(q_{C_1}) \ar[r]^{h} \ar[d]^{\Pi_1} & \RT(C) \ar[d]^{\Pi}   \\
 \bS_C \ar[r]^{\nabla} &  \bS_C \times_R \bS_C 
}\]
where the vertical maps are $\bAut(C)$--torsors.
Since the $\bAut(C)$-torsor $\Pi:\RT(C) \to \bS_C \times_R \bS_C$ is 
Zariski locally 
trivial according to Theorem \ref{th_zar}, it follows that
the $\bAut(C)$-torsor 
$\Pi_1 :\mathbf{Spin}(q_{C_1}) \to \bS_C$
is also Zariski locally trivial.

\begin{Rk} The map $\Pi_1$ is defined by the identification
$\mathbf{Spin}(q_{C_1}) \simlgr \RT(C)^\nabla$. As far as we know, there is no 
easier way to define it.
\end{Rk}

 The next result completes Theorem \ref{thm_main}.
 
\begin{Thm}\label{th_codiag}  There is a natural bijection
$$
 \mathbf{Spin}(q_{C_1})(R)  \backslash \bS_C(R)  \simlgr \ker\Bigl( 
H^1_{\mathrm{Zar}}\bigl(R, \bAut(C)\bigr) \to 
 H^1_{\mathrm{Zar}}\bigl(R, \mathbf{Spin}(q_{C_1})\bigr) \Bigr).
$$
It maps a point $a \in \bS_C(R)$ to the isomorphism class of the isotope 
$C^{a,a^{-1}}\simeq C^{1,a}$.
\end{Thm}

\begin{proof} We proceed is a similar, but simpler, way in comparison to the 
proof of Theorem \ref{th_zar}. There is a natural bijection \cite[prop. 
2.4.3]{G3},
$$
 \mathbf{Spin}(q_{C_1})(R)  \backslash \bS_C(R)  \simlgr \ker\Bigl( 
H^1_\fppf\bigl(R, \bAut(C)\bigr) \to 
 H^1_\fppf\bigl(R, \mathbf{Spin}(q_{C_1})\bigr) \Bigr).
$$
It maps a point $a \in \bS_C(R)$ to the class of the $\bG$--torsor 
$\Pi_1^{-1}(a) = \Pi^{-1}(a,a^{-1})$
which represents the isomorphism class of the isotope $C^{a,a^{-1}}$ by Theorem 
\ref{thm_main}.
Since the torsor $\Pi_1$ is locally trivial for the Zariski topology, 
we can replace the flat cohomology by the Zariski topology on the right hand 
side.  
\end{proof}

\subsection{Orthogonal and adjoint variants}
We have seen that the  centre of $\RT(C)$ is isomorphic to $\bm{\mu}_2 \times 
\bm{\mu}_2$ by the map
$(\eta_1, \eta_2) \mapsto ( \eta_1 \eta_2, \eta_1 , \eta_2)$ (before  Lemma 
\ref{lker}).
We consider now one copy $\bm{\mu}_2$ inside $\RT(C)$, namely the image of $\eta 
\mapsto ( 1, \eta , \eta)$. We will refer to this as the
\emph{diagonal $\bm{\mu}_2$}, as it induces a diagonal embedding of $\bm{\mu}_2$ 
in $\bS_C^2$ via $\Pi$.

We mod out the mapping $\Pi: \RT(C) \to \bS_C \times_R \bS_C$ by the action
of the diagonal $\bm{\mu}_2$.
This gives rise to the orthogonal variant of $\Pi$, namely 
\[\Pi_{+}: \RT(C)_{+}= \RT(C)/ \bm{\mu}_2 \to \bigr({\bS}_C \times_R  {\bS}_C \bigr)/ \bmu_2.\] 
 The $R$--map $f_1$                                                               
induces an $R$--isomorphism $\RT(C)_{+} \simlgr  \mathbf{SO}(q_C)$.
We thus have a Cartesian diagram
\begin{equation}\label{**}
 \xymatrix@1{
 \RT(C)  \ar[r]^{f_1} \ar[d]^{\Pi} & \mathbf{SO}(q_C) \ar[d]^{\Pi_+}  \\
{\bS}_C^2 \ar[r] & \bigl( {\bS}_C \times_R  {\bS}_C \bigr)/ \bmu_2.
}\end{equation}
where the vertical maps are $\bG$--torsors. Let $\ubS_C=\bS_C/\bm{\mu}_2$ (see 
the appendix for further discussion on quotients of spheres). The above provides 
us with two projections $p_1, p_2 : \bigl( {\bS}_C \times_R  {\bS}_C \bigr)/ 
\bmu_2 \to 
{\ubS}_C$ and, correspondingly, two actions of
$\mathbf{SO}(q_C)$ on $\ubS_C$, which we denote by $g\bullet_i x$ for $i = 1, 
2$. An important step is the following.

\begin{Lma}\label{lem_section}
\begin{enumerate}
 \item For $i = 1, 2$,  the orbit map $\mathbf{SO}(q_C) \to  \ubS_C$,  $g 
\mapsto  g\bullet_i  [1]$ admits a
splitting.
\item Both actions of $\mathbf{SO}(q_C)(R)$ on $\ubS_C(R)$ are transitive.
\end{enumerate}
\end{Lma}

\begin{proof}
 We will only do the case $i = 1$, the case $i = 2$ being similar.
 \smallskip
 
\noindent(1) We consider the $R$--map $\bS_C \to \mathbf{SO}(q_C)$ which maps an 
element  $a \in  \bS_C(R)$
to the orthogonal mapping $B_a$. Since this map is $\bmu_2$--invariant, it gives 
rise to an $R$--map
$h : \ubS_C \to   \mathbf{SO}(q_C)$ . We claim that $x\mapsto h(x)^{-1}$ is a 
section of the orbit map. 
Let $ x \in  \ubS_C(R)$ and let $S/R$ be a flat cover such that $x_S$ lifts to 
an element $a \in  \bS_C(S)$.
We consider $h(x) : C \to  C$. Then $h(x)_S = B_a : C_S \to C_S$ which lifts 
(with respect to $f_1$) to the element $\mathbf{t} = (B_a,R_{\overline 
a},L_{\overline a}) \in \RT(C)(S)$.
We compute in $\ubS_C(S)$
$$
h(x)\bullet_1 x_S = p_1(\mathbf{t}. x_S) = [L_{\overline{a}}(a)] = [1] 
$$
whence $h(x)^{-1}. [1] = x$. Thus $x\mapsto h(x)^{-1}$ is a section of the orbit 
map.

\smallskip

\noindent (2) This is immediate from (1).

\end{proof}

One can also go further and consider the quotient of $\RT(C)$ by the whole 
centre. We observe that the $\bAut(C)$--torsor $\Pi: \RT(C) \to \bS_C^2$ is 
$\bm{\mu}_2 \times \bm{\mu}_2$--equivariant 
for the antipodal action of $\bm{\mu}_2 \times \bm{\mu}_2$  on $\bS_C^2$. 
Taking the quotient with respect to this action, we get an $R$--morphism
$$
\underline{\Pi}:  \underline{\RT}(C) \to \underline{\bS}_C^2
$$
where $\underline{\RT}(C) = \RT(C) / (\bm{\mu}_2 \times \bm{\mu}_2)$ is the 
adjoint $R$--group scheme of
$\RT(C)$, which is isomorphic to $\mathbf{PSO}(q_C)$. We consider now the 
$R$--homomorphism
\[\underline{i}: \bAut(C) \xrightarrow{i} \RT(C)  \to \underline{\RT}(C).\]

\begin{Lma} 
\begin{enumerate}
\item The $R$--homomorphism $\underline{i}$ is a closed immersion and the 
fppf quotient
$\underline{\RT}(C)/ \underline{i}(\bAut(C))$ is representable by an affine 
$R$--scheme. 
\item The $R$--morphism $\underline{\Pi}:  \underline{\RT}(C) \to 
\underline{\bS}_C^2$ is an $\bAut(C)$--torsor
and the map $\underline{\RT}(C)/ \underline{i}(\bAut(C)) \to \underline{\bS}_C^2$ 
is an $R$--isomorphism.
\end{enumerate}
\end{Lma}

\begin{proof} The $R$--map ${\bS}_C^2 \to \underline{\bS}_C^2$ is faithfully 
flat.
The main point is that taking the quotient by $\bm{\mu}_2 \times \bm{\mu}_2$ 
commutes with base change so that the left hand square of the
following diagram
\begin{equation}\label{***}
\xymatrix@1{
\RT(C)  \ar[r]^{P} \ar[d]^{\Pi} & \underline{\RT}(C) \ar[d]^{\underline{\Pi}} 
\ar[r]_{\sim}^{\underline{f_1}} & \mathbf{PSO}(q_C) \ar[dl]^{\underline{\Pi}'}\  
\\
{\bS}_C^2 \ar[r] & \underline{\bS}_C^2.
} 
\end{equation}
is Cartesian. The $R$--group scheme $\bAut(C)$ acts on the map $P$.
Since the property for that action to define an $\bAut(C)$--torsor is 
insensitive to a flat cover, $\Pi$ is an $\bAut(C)$-torsor if and only if 
$\underline{\Pi}$ is.
Since $\Pi$ is an $\bAut(C)$--torsor, so is $\underline{\Pi}$, and this implies 
that
$\underline{i}$ is an $R$--monomorphism and that  the fppf-quotient sheaf  
$\underline{\RT}(C)/ \underline{i}(\bAut(C)) $
is representable by the affine smooth $R$--scheme $\underline{\bS}_C^2$. Then 
$\bAut(C)$ occurs as the fibre of 
$\underline\Pi$ at $[(1,1)] \in \underline{\bS}_C^2(R)$ so $\underline{i}$ is a 
closed $R$--immersion.
\end{proof}

\subsection{Codiagonal adjoint variant}\label{scodiag}
Finally, one may combine the above and consider the quotient of 
$\RT(C)^\nabla\simeq \mathbf{Spin}(q_{C_1})$ by its centre. The centre of 
$\RT(C)^\nabla$ is $\bm{\mu}_2$, embedded
as $\eta\mapsto(1,\eta,\eta)$ for each $R$--ring $S$ and $\eta\in\bm{\mu}_2(S)$, 
and Lemma \ref{lem_nabla} implies that $\RT(C)^\nabla/\bm{\mu}_2\simeq 
\mathbf{SO}(q_{C_1})$. Reasoning as in the proof of as in the preceding cases, 
we take the quotient of the torsor 
$\Pi_1:\mathbf{Spin}(q_{C_1})\to\bS_C$ by $\bm{\mu}_2$, thus obtaining the 
Cartesian diagram
\[\xymatrix@1{
\mathbf{Spin}(q_{C_1}) \ar[r] \ar[d]^{\Pi_1} & \mathbf{SO}(q_{C_1}) 
\ar[d]^{\underline\Pi_1}   \\
 \bS_C \ar[r]&  \ubS_C
.}\]
The vertical maps are $\bAut(C)$--torsors; indeed this has already been seen for 
$\Pi_1$, and for $\underline{\Pi}_1$ it follows from the corresponding statement 
for $\underline{\Pi}$ in the above lemma. In fact we have the commutative diagram
\[\xymatrix@1{
\mathbf{Spin}(q_{C_1}) \ar[r] \ar[d]^{\Pi_1} & \mathbf{SO}(q_{C_1}) 
\ar[d]^{\underline\Pi_1} \ar[r] & \mathbf{SO}(q_{C}) \ar[d]^{\Pi_+}   \\
 \bS_C \ar[r]&  \ubS_C \ar[r]^{\underline{\nabla}\qquad}&\bigl( {\bS}_C \times_R  
{\bS}_C \bigr)/ \bmu_2
}\]
where $\underline{\nabla}$ is induced by $\nabla:\bS_C\to\bS_C\times_R\bS_C$ and 
is well-defined since $\nabla(\eta x)=\eta\nabla(x)$ for all $R$--rings $S$, 
$\eta\in\bm{\mu}_2(S)$ and $x\in\bS_C(S)$.

\subsection{Equivalence of all variants}
The following theorem shows that all variants considered above are equivalent.

\begin{Thm} \label{thm_kernel}
For $i=1,\ldots,8$, let
\[N_i=\Ker\bigl( H^1_\fppf(R, \bG) \to H^1_\fppf(R, \bH_i) \bigr),\]
where
\[\begin{array}{llll}
\bH_1=\mathbf{Spin}(q_{C_1}),& \bH_2=\mathbf{Spin}(q_{C}),& 
\bH_3=\mathbf{SO}(q_{C}),& \bH_4=\mathbf{O}(q_{C}),\\
\bH_5=\mathbf{SO}(q_{C_1}),& \bH_6=\mathbf{PSO}(q_{C}),& 
\bH_7=\mathbf{GO^+}(q_{C}),& \bH_8=\mathbf{GO}(q_{C}),
  \end{array}
\]
Then all the sets $N_i$ coincide and are subsets
of $H^1_{\mathrm{Zar}}(R,\bG)$.
\end{Thm}

\begin{proof}
We have the following inclusions of sets
\[\xymatrix@1{
N_1 \ar[r]\ar[dr]& N_2 \ar[r]& N_3 \ar[r]\ar[dr]\ar[d]& N_4 \ar[r]& N_8 \ar[r]& 
H^1_\fppf(R, \bG)\\
    & N_5\ar[ur]    & N_6 & N_7 \ar[ur]
}\]
and will show the reverse inclusions in three steps.

\emph{The sets $N_1,\ldots,N_5$.} We have already seen along the proof of 
Theorem \ref{th_zar} that $N_4$ consists of Zariski
classes. According to Theorem \ref{th_zar} (resp.\ Th. \ref{th_codiag}), $N_2$ 
(resp.\ $N_1$)
consists of isomorphism classes of isotopes $C^{a,b}$ with $a, b \in  \bS_C(R)$ 
(resp.\
$C^{a, a^{-1}}$ with $a \in  \bS_C(R)$). Then Proposition \ref{prop_formulae} 
shows that $N_1=N_2$. On the other hand, the octonionic involution $\sigma_C$ provides a
splitting of the Dickson homomorphism $\mathbf{O}(q_C) \to  \Z/2\Z$. 
This implies that the map $H^1_\fppf(R, \mathbf{SO}(q_C)) \to H^1_\fppf(R, 
\mathbf{O}(q_C))$
 has trivial kernel, whence $N_3=N_4$. Next we show that the $N_3\subseteq N_2$. 
The characteristic map $\varphi$ of the $\bG$--torsor $\Pi_+$
induces the  bijection
$$
\mathbf{SO}(q_C)(R) \backslash \bigl( ({\bS}_C \times_R  {\bS}_C)/ \bmu_2 
\bigr)(R)
\simlgr N_3.$$
Let $[C'] \in N_3$ and
let $x \in \bigl( ({\bS}_C \times_R  {\bS}_C)/ \bmu_2 \bigr)(R)$ be
such that $\varphi(x) = [C']$. We denote by $x_1$ its first projection on 
$\ubS_C(R)$.
Lemma \ref{lem_section}.(2)  shows that
the group $\mathbf{SO}(q_{C})(R)$ acts transitively on $\bS_C(R)$. Up to 
replacing $x$ by a suitable 
$\mathbf{SO}(q_{C})(R)$--conjugate, we may thus assume that $x_1 = [1] \in  
\ubS_C(R)$. The commutative diagram of
$\bmu_2$--torsors
\[\xymatrix@1{
{\bS}_C \times_R  {\bS}_C  \ar[r]^{p_1} \ar[d] &  \ar[d] \bS_C  \\
 \bigl( {\bS}_C \times_R  {\bS}_C \bigr)/ \bmu_2 \ar[r]^{\qquad \underline{p}_1} 
&  {\ubS}_C.
}\]
shows that $x$ lifts to an element $(a, b) \in  \bS_C^2(R)$. The diagram 
\eqref{**} shows that
$\Pi^{-1}(a,b)= \Pi^{-1}_+(x)$ whence 
$C^{a, b} \simeq C'$.  Thus $[C']\in N_2$ as desired. Finally, the statement for 
$N_5$ follows from the inclusions $N_1\subseteq N_5\subseteq N_3$.

\smallskip

\noindent \emph{The set $N_6$.} To show that $N_6\subseteq N_3$, we consider the bijection
\[\mathbf{PSO}(q_C)(R) \backslash {\ubS}_C(R)^2\simlgr N_6\]
induced by the characteristic map $\varphi'$ of the torsor $\underline{\Pi}'$ 
from \eqref{***}. Let $[C']\in N_6$ and let $(y,z)\in\ubS_C^2$ be such that 
$\varphi'(y,z)=[C']$.
We need show that the $\mathbf{PSO}(q_C)(R)$--orbit of $(y,z)$ contains an 
element in the image of $\bS_C(R)^2$ under the quotient map.
From Lemma \ref{lem_section}, this orbit contains $(x,1)$ for some  
$x\in\ubS_C(R)$. 
We now consider the $R$--map $\bS_C\to\mathbf{SO(C_q)}$ defined by $a\mapsto 
L_{\overline{a}}$ for any $R$--ring $S$ and $a\in\bS_C(S)$. This map commutes 
with the $\bmu_2$--action
since for any $a\in\bS_C(S)$ and any $\eta\in\bmu_2(S)$,
\[\eta a\mapsto L_{\overline{\eta a}}=L_{\eta\overline{a}}=\eta 
L_{\overline{a}},\]
and thus it induces an $R$--map $g:\ubS_C\to \mathbf{PSO}(q_C)$. Following the 
strategy of Lemma \ref{lem_section}, given $x\in\ubS_C(R)$, we pick a flat cover 
$S/R$ such that $x_S=p_S(a)$
for some $a\in\bS_C(S)$, where $p$ denotes the quotient map $\bS_C\to\ubS_C$. 
Then with the notation of the diagram \eqref{***}, 
\[g(x)_S=\underline{f_1}\circ P(L_{\overline{a}},B_a,R_{\overline{a}}).\]
Computing the action of $g(x)$ on $(x,1)$ in $\ubS_C(S)$, we have
\[(g(x).(x,1))_S=(p_S(R_{\overline{a}}(a),p_S(B_a(1)))=(p_S(1),p_S(a^2))=(1,x_S^2).\]
Thus $g(x).(x,1)=(1,x^2)$. Finally, $x^2$ is in the image of $p_R:\bS_C(R)\to 
\ubS_C(R)$. To see this, consider the map $a\mapsto a^2$ on $\bS_C$. Since it is 
the trivial map
on $\bmu_2$, it induces a map $f:\ubS_C\to\bS_C$. Thus $f(x)\in\bS_C(R)$. But 
then $x^2=p\circ f(x)$ since
\[p_S\circ f_S(x_S)=p_S(a^2)=p_S(a)^2=x^2.\]
Thus under the action of $\mathbf{PSO}(q_C)$, the 
element $(y,z)$ is conjugate to the image of some $(a,b)\in\bS_C(R)^2$, whence 
$C'$ is isomorphic to $C^{a,b}$ and thus $[C']\in N_2$ as desired.

\smallskip

\noindent \emph{The sets $N_7$ and $N_8$.} We are given  $[C'] \in  N_8$.
The exact sequence 
\[\mathbf{1} \to \mathbf{O}(q_C) \to \mathbf{GO}(q_C) \to \mathbf{G}_m \to 
\mathbf{1}\]
gives rise to an exact sequence of pointed sets
$$
R^* \to H^1_\fppf(R, \mathbf{O}(q)) \to H^1_\fppf(R, \mathbf{GO}(q))  .
$$
Since $[q_{C'}] \in H^1_\fppf(R, \mathbf{O}(q))$ maps to $1 \in H^1_\fppf(R, 
\mathbf{GO}(q))$, there exists $\lambda \in 
R^*$ and an isomorphism   $f: (C',q_{C'}) \simlgr 
(C, \lambda q_C)$. Since  $q_{C'}$ represents $1$, $q_C$ represents $\lambda$, 
i.e. there exists $x \in C^*$
such that $\lambda = q(x)$. Up to replacing $f$ by $L_x^{-1} \circ f$, we may 
assume that 
$\lambda=1$ since $q_C$ is a multiplicative $R$--form. It follows that the 
quadratic spaces $(C',q_{C'})$ and 
$(C, q_C)$ are isomorphic. Thus $[C']\in N_4$ and is in particular a Zariski 
class.
It follows that $N_8=N_4$, and since $N_4\subseteq N_7\subseteq N_8$ we get 
$N_7=N_4$ as well.
\end{proof}

Our main findings can thus be summarized as follows.

\begin{Cor} Let $C$ and $C'$ be octonion algebras over $R$. The following statements are equivalent.
\begin{enumerate}
\item The quadratic forms $q_C$ and $q_{C'}$ are isometric.
\item The quadratic forms $q_C$ and $q_{C'}$ are similar.
\item There exist $a,b\in \bS_C(R)$ such that $C'$ is isomorphic to $C^{a,b}$.
\item There exists $a\in \bS_C(R)$ such that $C'$ is isomorphic to $C^{a,a^{-1}}=C^{a,\overline{a}}$.
\end{enumerate}
\end{Cor}

\begin{proof} The set $N_4$ classifies the forms of $C$ whose norm is isometric  
to $q_C$, while the set $N_8$ classifies those whose norm is similar  
to $q_C$. The first two items are thus equivalent by Theorem \ref{thm_kernel}.
Theorem \ref{th_zar} gives the equivalence of (2) and (3), while the last two statements are equivalent by 
Proposition \ref{prop_formulae}.
\end{proof}

\begin{Rk} This statement generalizes the result of \cite{KPS} quoted in 
Proposition \ref{prop_KPS} above. It also recovers Remark 
\ref{Rreduce}: given $a,b\in C^*$, there exists $c \in \bS_C(R)$ such that 
$C^{a,b}$ is isomorphic to $C^{1,c}\simeq C^{c,c^{-1}}$. 
Indeed,  we have seen in Remark \ref{Rcomp} that the norm of the isotope $C^{a,b}$ is isometric to 
$q_C$, and the previous Corollary applies.
\end{Rk}

\begin{Cor} Let $C'$ be a $R$-form of $C$ such that
there exists a non-singular quadratic $R$-module $(\calL, \theta)$ of rank one
such that the quadratic $R$-form $(C', q_{C'})$ is isometric
to the quadratic $R$-form   $(\calL, \theta) \otimes (C, q_{C})$.
Then $C'$ is an isotope of $C$. In particular
$(C', q_{C'})$ is isometric to $(C, q_{C})$.
\end{Cor}

\begin{proof}
We consider the exact sequence of $R$-group schemes
$$
1 \to \bmu_2  \xrightarrow{i}  \mathbf{SO}(q_C) \to \mathbf{PSO}(q_C)
\to 1.
$$
It gives rise to the exact sequence of pointed sets
$$
H^1_\fppf(R, \bmu_2)  \xrightarrow{i_*} 
H^1_\fppf(R,\mathbf{SO}(q_C)) \to H^1_\fppf(R,\mathbf{PSO}(q_C)).
$$
On the other hand, the  exact sequence $ 1  \to \mathbf{SO}(q_C) \xrightarrow{j_*}  \mathbf{O}(q_C) \to \Z/2\Z \to 1$
(which is split by $\kappa$) induces a map $H^1_\fppf(R,\mathbf{SO}(q_C) \bigr) \to H^1_\fppf(R, \mathbf{O}(q_C)\bigr)$.
As explained in the appendix, there is a group isomorphism $\Disc(R) \simeq  H^1_\fppf(R, \bmu_2)$ between
the group of isomorphism class of non-singular quadratic forms of  rank one and   $H^1_\fppf(R, \bmu_2)$.
The composite map 
\[\Disc(R) \simeq  H^1_\fppf(R, \bmu_2) \xrightarrow{i_*} 
H^1_\fppf(R,\mathbf{SO}(q_C)) \xrightarrow{j_*}  H^1_\fppf(R,\mathbf{O}(q_C))\]
maps the class of a quadratic $R$--form 
$(\calL, \theta)$ of rank one to the class of the quadratic form
$(\calL, \theta) \otimes (C, q_{C})$.
The  hypothesis  that $(C', q_{C'})$ is similar
to $(C, q_{C})$ implies   that $[C']$  belongs to the kernel of the map
 $$
 H^1_\fppf(R, \bG) \to H^1_\fppf(R, \mathbf{PSO}(q_C))
 $$
 Theorem \ref{thm_kernel} states that $[C']$  belongs to all
 relevant kernels and therefore is an isotope.
\end{proof}

A complement is the following. 

\begin{Prp}\label{prop_marlin}  With the notations of Theorem \ref{thm_kernel}, 
each flat quotient $\bH_i / \bG$ is representable by an affine $R$-scheme of finite presentation ($i=1,...,8$).
Furthermore the quotient map $\bH_i \to \bH_i/\bG$ is a $\bG$--torsor locally trivial 
for the Zariski topology. If $R \not 0$, the $\bG$--torsor $\bH_i \to \bH_i/\bG$ is non trivial ($i=1,...,8$), that is 
does not admit a section. 
\end{Prp}

Note that we have already seen the first part of the statement in several cases. 

\begin{proof}
The representability item boils down to the case of $\Z$ and of the Zorn algebra.
Representability then follows from Anantharaman's theorem \cite[Th. 4C page 53]{An};
the fact that the quotient is affine of finite presentation and also the fact that the 
quotient map is a $\bG$-torsor are due to Colliot-Th\'el\`ene and Sansuc \cite[Cor. 6.12]{CTS}
for the connected $\bH_i$'s. This generalizes easily to the non-connected cases (i.e. $\bH_4$ and $\bH_8$)
by embedding  $\bH_i$ in some $\mathbf{GL}_{n_i}$.  The fact that those $\bG$--torsors are trivial locally for the Zariski topology
follows  of Theorem \ref{thm_kernel} applied to the coordinate rings $R[\bH_i/\bG]$.

Assuming that $R$ is a non-zero ring, we want to  show that the $\bG$-torsor \break $\bH_i \to \bH_i/ \bG$ is non-trivial.
We can then assume that $R=k$ is an algebraically closed field and would like to find a $k$-ring $S$
such that $H^1_\fppf(S, \bG) \to H^1_\fppf(S, \bH_i)$ is non-trivial. 
 Theorem \ref{thm_kernel} enables us to choose one of the $\bH_i$'s and we consider then the case
 of $\bH_5= \mathbf{SO}(q_{C_1})$. Using the beginning of Section \ref{scodiag}, 
 we know that $\mathbf{SO}(q_{C_1})/ \bG \simlgr \ubS_q$. We take $S=k[\ubS_q]$, i.e.\ the coordinate ring of $\ubS_q$, and it is enough to 
 show, using the universal point $\eta \in \ubS_q(S)$, that 
 the $\bG$-torsor  $\mathbf{SO}(q_{C_1}) \to  \ubS_q$ is non-trivial.
 
 According to Lemma  \ref{lem_quadric}, $\ubS_q$ is the complement in $\mathbb{P}(C^\vee)$ of the projective
 quadric $Q$ of equation $q_C=0$. We now use the theory of cycles and Chow groups.
 We have a localization  exact sequence
 \cite[\S 1.8]{F}
 $$
  \Pic(Q)    \xrightarrow{i_*}   \mathrm{CH}^2( \mathbb{P}(C^\vee) ) \to    \mathrm{CH}^2(\ubS_q) \to 0 .
 $$
 Let $h$ be a hyperplane section of $\mathbb{P}(C^\vee)$.
We have $\mathrm{CH}^2( \mathbb{P}(C^\vee) ) = \Z h^2$ and we know that the hyperplane section $h \cap Q$ generates 
$\Pic(Q)$. But $i_*i^*(h)= h . [Q]$ where $[Q] \in \Pic( \mathbb{P}(C^\vee) )$ ({\it ibid}, prop. 2.6.(b)).
Finally we have $[Q]= 2h$ ({\it ibid}, Example 1.9.4)  which allows us to conclude that $\mathrm{CH}^2(\ubS_q)$ is isomorphic to $\Z/2\Z$.
On the other hand, Marlin showed that $\mathrm{CH}^2(\mathbf{SO}(q_{C_1}))=0$ \cite{Ma} since $\mathbf{SO}(q_{C_1})$ is an
adjoint semisimple group of type $B_3$. 
The functoriality of Chow groups yields that the map $\mathbf{SO}(q_{C_1}) \to \ubS_q$ cannot admit a splitting.
\end{proof}

 \begin{Rk} The advantage of this proposition compared to Example \ref{ex_mimura} is that it is algebraic
and works, therefore, in any characteristic. Also it shows that some 2-torsion phenomenon is involved
while   Example \ref{ex_mimura} was based on 3-torsion. In their paper \cite{AHW}, Asok--Hoyois--Wendt
 investigate the 2-torsion and 3-torsion phenomena in the split case by means of 
 cohomological invariants.
\end{Rk}

\section{Particular cases}\label{sec_particular}
We have seen above that given an octonion algebra $C$ over $R$, any octonion 
algebra with quadratic form equivalent to that of $C$ is isomorphic to 
$C^{1,a}\simeq C^{a,\overline{a}}$
for some $a\in \bS_C(R)$. Two natural questions arise. On the one hand, one may 
ask for criteria on $a$, independent of the base ring, in order that 
$C^{1,a}\simeq C$. This supplements the relations among isotopes of Section 
\ref{sec_background} and Remark \ref{rk_rel}. On the other
hand one may wish to understand the situation over rings of particular interest. 
The aim of this section is to discuss these questions.

\subsection{Trivial isotopes}\label{striv}
 A question that arises in this context is whether one can characterize those 
$a$ for which $C^{1,a}\simeq C$. In this section we give sufficient conditions
for this to hold. They are consequences of the second item of the following proposition, where, for 
any $c\in C^*$, we define the map $\tau^+_c:C\to C$ by $x\mapsto cxc^2$ and 
$\tau^-_c:C\to C$ by $x\mapsto c^{-2}xc^{-1}$. The first item, while not necessary for the sequel, is included 
for completeness. 

\begin{Prp} Let $a\in\bS_C(R)$.

\begin{enumerate}
\item If there exist $c_1,\ldots,c_r\in C^*$ such that $L_{c_r}\cdots 
L_{c_1}(1)=L_{\overline{c_r}}\cdots L_{\overline{c_1}}(1)=a$, then
\[B_{\overline{c_r}}\cdots B_{\overline{c_1}}: C\to C^{a,\overline{a}}\]
is an isomorphism of algebras.
\item If there exist $c_1,\ldots,c_r\in C^*$ such that 
$a=\tau^+_{c_r}\cdots\tau^+_{c_1}(1)$, then
\[L_{c_r}R_{c_r}^{-1}\cdots L_{c_1}R_{c_1}^{-1}: C\to C^{a,\overline{a}}\]
is an isomorphism of algebras. 
\end{enumerate} \end{Prp}

\begin{Rk} In \cite{Eld}, Elduque proves that over a field $K$, the image of the 
map $f_1:\RT(C)(K) \to \mathbf{SO}(q_C)(K)$ is precisely
\[\left\{\prod_{i=1}^r B_{d_i} | r\in\mathbb N, d_i\in C^*, \prod_{i=1}^r 
q(d_i)=1\right\}.\]
When this holds, the first item above gives, together with Proposition 
\ref{Ptrialiso}, a precise condition for the existence of an isomorphism from 
$C$ to $C^{a,\overline{a}}$. 
\end{Rk}

\begin{proof} 
(1) (See also \cite{Pe2}.) The hypothesis implies that $R_{c_r}\cdots R_{c_1}(1)=\overline{a}$ and that 
$q(c_1)\cdots q(c_r)=1$. Thus 
\[(B_{\overline{c_r}}\cdots B_{\overline{c_1}},R_{c_r}\cdots 
R_{c_1},L_{c_r}\cdots L_{c_1})\in\RT(C),\]
and the claim follows from Proposition \ref{Ptrialiso}.

\smallskip

(2) Using the Moufang laws, one verifies that for each $c\in C^*$ we have 
\begin{equation}\label{conjug}
L_cR_c^{-1}(\overline{x}\ \overline{y})=\tau^+_c(\overline{x})\tau^-_c(\overline{y})=\overline{\tau^-_c(x)}\ \overline{\tau^+_c(y)} .  
\end{equation}
If $c_1,\ldots,c_r\in C^*$ are such that 
$\tau^+_{c_r}\cdots\tau^+_{c_1}(1)\in\bS_C(R)$, then $q(c_1)^3\cdots 
q(c_r)^3=1$, whence \eqref{conjug} implies that 
\[(L_{c_r}R_{c_r}^{-1}\cdots 
L_{c_1}R_{c_1}^{-1},\tau^-_{c_r}\cdots\tau^-_{c_1},\tau^+_{c_r}\cdots\tau^+_{c_1
})\in\RT(C)(R).\]
We then conclude with Proposition \ref{Ptrialiso}.
\end{proof}

This has the following consequences.

\begin{Cor}\label{cor_cube} Let $a\in \bS_C(R)$. If either
\begin{enumerate}
\item there exists $c\in C^*$ with $a=c^3$, or
\item there exists $u\in C^*$ such that $b_q(u,1)=b_q(u,a)=0$, or
\item $b_q(a,1)=0$,
\end{enumerate}
then $C^{a,\overline{a}}\simeq C$.
\end{Cor}

\begin{proof} The first item follows from item (2) of the above proposition upon 
setting $r=1$ and $c_1=c$. For the second, we set instead $r=2$, $c_1=u^{-1}$ and 
$c_2=au$, and note that $u\overline{a}=au$. Then 
$c_2^2=(u\overline{a})(au)=u^2$ implies that
\[\tau^+_{c_2}\tau^+_{c_1}(1)=(au)(u^{-3}(au)^2)=(au)(u^{-3}u^2)=(au)u^{-1}=a.\]
The third item follows from the first since if $b_q(a,1)=0$, then 
$a^2=-q(a)1=-1$ and $a=(-a)^3$. 
\end{proof}

As a consequence of the second item, we easily obtain the well-known fact that over fields, octonion algebras with equivalent quadratic forms are isomorphic.

\begin{Cor} Let $k$ be a field and let $C$ and $C'$ be two octonion algebras
over $k$ with isometric norms. Then $C\simeq C'$.
\end{Cor}

\begin{proof} It suffices to prove that $C^{a,a^{-1}}\simeq C$ for any
$a\in\bS_C(k)$. By the above corollary, it moreover suffices to show that there exists an invertible
$u\in a^\bot\cap 1^\bot$. But Lemma \ref{lem_summand} implies that $a^\bot\cap 1^\bot$ is a subspace of dimension at least 6. By Witt's theorem, 
the regularity of $q_C$ implies the existence of $u\in a^\bot\cap 1^\bot$ with $q_C(x)\neq 0$, and the claim follows.
\end{proof} 

\begin{Rk} It is possible to extend this argument to cover the case of octonion algebras over local rings. This requires some technical work,
including a generalization of Corollary \ref{cor_cube}. We therefore refrain from carrying it out, referring instead to the proof found in \cite{B}.
\end{Rk}

\subsection{Isotopes  of the Zorn algebra}

We shall now slightly generalize a result by Asok--Hoyois--Wendt \cite[Prop. 4.3.4]{AHW}.
We assume in this subsection that $C$ is the split octonion algebra.
It admits a subalgebra $R \times R$ and we know that 
its centralizer in $\bG$  is the standard $R$--subgroup $\mathbf{SL}_3 \subset \bG$ \cite[Theorem 5.9]{LPR}.

\begin{Prp}\label{prop_zorn} Assume that $C$ is the split octonion algebra. We have
$$
\ker\Bigl( H^1_{Zar}( R, \bG) \to H^1_{Zar}( R, \mathbf{Spin}(q_C) )\Bigr)
\subseteq \mathrm{Im}\Bigl( H^1_{Zar}( R, \mathbf{SL}_3) \to H^1_{Zar}( R, \bG )\Bigr)
$$
\end{Prp}

\begin{proof} We reason in terms of octonion algebras. Let $[C']$ be an element of the 
leftmost kernel, i.e.\ the class of an isotope $C'$ of $C$.
There exists an isometry \break $f: (C, q_C) \simlgr (C',q_{C'})$ mapping $1_C \to 1_{C'}$.
We put $A'=f(R \times R)$. Then $A'$ is a composition $R$--subalgebra of $C'$. 
We consider now the flat sheaf defined by
$$
F(S)= \Bigl\{  g: C_S \simlgr C'_S  \, \mid \, \hbox{$g$ induces $f_S: (R \times R)_S \simlgr A'_S$ }\Bigr\} 
$$
for each $R$-ring $S$. It is clearly representable by an affine $R$--scheme $\bY$ and 
comes equipped with a right action of $\mathbf{SL}_3$
so that the $R$--map $\bY \times_R \mathbf{SL}_3 \to \bY \times_R \bY$, $(y,h) \mapsto (h, y.h)$ is an isomorphism. 
Inspection of the proof of Bix's theorem \cite[lemma 1.1]{B} shows that $\bY(R)$ is non empty when $R$ is local.
Altogether it follows that  $\bY$ is a $\mathbf{SL}_3$-torsor locally trivial for the Zariski topology. 
The $\bG$--torsor $\mathrm{Isom}(C,C')$ then arises as the contracted product 
of the $\mathbf{SL}_3$-torsor $\bY$ with respect to the embedding $\mathbf{SL}_3 \hookrightarrow \bG$. 
\end{proof}

\medskip

The map $H^1_\fppf( R, \mathbf{SL}_3) \to H^1_{\fppf}( R, \bG )$ is well understood
in terms of octonion algebras. The set  $H^1_\fppf( R, \mathbf{SL}_3)$ classifies
 the pairs $(P, \theta)$ where $P$ is a locally free $R$-module of rank $3$
 and $\theta:R \simlgr \Lambda^3(P)$ a trivialization of its determinant.  
 Petersson \cite[\S 3]{Pe} explicitly constructed the $R$--algebra $\mathrm{Zorn}(P, \theta)$ associated to 
 $(P, \theta)$  and its norm form is the hyperbolic form attached to the $R$--module
 $R \oplus P$. In conclusion, we have shown that an isotope of $C$ is an $R$--algebra $\mathrm{Zorn}(P, \theta)$ 
 such that the hyperbolic form attached to  $R \oplus P$ is isomorphic to that of $R^4$.
 
 \begin{Rk}
  (a) The original statement \cite[Prop. 4.3.4]{AHW} deals with the case of a smooth ring over an 
  infinite field and uses $\mathbf{A}^1$-homotopy theory. Our proof is then in a sense much simpler. 
 
 \smallskip
 
 \noindent (b) The map  $H^1_\fppf( R, \mathbf{SL}_3) \to H^1_{\fppf}( R, \mathbf{GL}_3 )$
 has trivial kernel. If $P$ is free, it follows that $(P, \theta)$ encodes the trivial class of
 $H^1_\fppf( R, \mathbf{SL}_3)$ so that  $\mathrm{Zorn}(P, \theta)$ is the split octonion
 $R$--algebra. In particular, if all projective modules of rank 3 are free, 
 Proposition \ref{prop_zorn} implies that all isotopes of the Zorn algebra are trivial.
 \end{Rk}

 Using Swan's extension of Quillen--Suslin's theorem \cite[Cor. 1.4]{Sw2}, 
 we get the following fact. (In the polynomial case, this was already proved in \cite[Prop.\ 7.4]{KPS}.) 
 
 \begin{Cor}\label{cor_zorn} If $k$ is a field and $R=k[t_1^{\pm 1},...,t_d^{\pm 1}, t_{d+1}, \dots ,t_{n}]$, then 
 the isotopes of the split octonion algebra are split. 
 \end{Cor}

\subsection{Polynomial and Laurent polynomial rings} 

We will now consider octonion algebras over polynomial and Laurent polynomial rings. We start with a preliminary fact about unit spheres.

\begin{Lma}\label{lem_aniso1} Assume that $R$ is  an integral domain with fraction field $K$.
Let $q$ be a non-singular  quadratic $R$--form such that $q_K$ is anisotropic.
Then  
$$
\bS_q(R)= \bS_q( R[t])= \bS_q(R[t^{\pm 1}]).
$$
\end{Lma}

\begin{proof}  
The assumption implies that the semisimple $K$--group $\mathbf{SO}(q_K)$ is anisotropic. Thus
$\mathbf{SO}(q)\bigl( K[[t^{-1}]] \bigr)=\mathbf{SO}(q)\bigl(K((t^{-1})) \bigr)$
by Bruhat--Tits--Rousseau's theorem (see \cite[Th. 1.2]{G1}).
Since $\mathbf{SO}(q)\bigl(K((t^{-1})) \bigr)$ acts transitively on $\bS_q\bigl(K((t^{-1})) \bigr)$, we obtain $\bS_q\bigl(K[[t^{-1}]] \bigr)= \bS_q\bigl(K((t^{-1})) \bigr)$.

We are given a point $x \in \bS_q(R[t])$; its image in $\bS_q\bigl(K((t^{-1})) \bigr)$
then belongs to $\bS_q\bigl(K[[\frac{1}{t}]] \bigr)$, so that the image $x_K$ of 
$x$ in $\bS_q(K[t])$ extends to a morphism $\mathbb{P}^1_K \to \bS_q \times_R K$.
It follows that $x_K$ belongs to $\bS_q(K)$. But $K \cap  R[t]= R$, whence
$x \in \bS_q(R)$. Similarly, working with $K((t))$, we find $\bS_q(R)= \bS_\varphi(R[t^{\pm 1}])$.
\end{proof}

\begin{Lma} \label{lem_aniso2} Assume that $R$ is  an integral domain  with fraction field $K$.
Let $C$ be an octonion $R$--algebra such that $C_K$ is not split.
Let $C'$ be a $R[t]$-isotope   of $C \otimes_R R[t]$.
Then $C' \simeq  B \otimes_R R[t]$ where $B$ is an $R$-isotope of $C$.  
The same holds with $R[t^{\pm 1}]$. 
\end{Lma}

\begin{proof} 
Our assumption is that $C_K$ is not split, equivalently the quadratic form
$q_C \otimes_R K$ is anisotropic. 
 Since isotopes of $C \otimes_R R[t]$ are parametrized by 
$\bS_C(R[t])$, the statement follows from
Lemma \ref{lem_aniso1}.
\end{proof}

Since anisotropy is preserved by purely transcendental extensions, we get by induction the 
following fact which  recovers and slightly extends a result by \break Knus--Parimala--Sridharan \cite[Cor. 7.2]{KPS}.

\begin{Cor}\label{cor_aniso} Let  $k$ be a field and assume that  $R=k[t_1^{\pm 1},...,t_d^{\pm 1}, t_{d+1}, \dots ,t_{n}]$,
Let $C$ be a non-split octonion $k$--algebra. 
Then all $R$--isotopes of $C\otimes_k R$ are trivial. 
\end{Cor}

In the case of Laurent polynomial rings over fields, we can generalize this statement as follows.

\begin{Prp}\label{prop_laurent} Assume that $R=k[t_1^{\pm 1}, \dots ,t_n^{\pm 1}]$ is 
a Laurent polynomial ring over a field $k$ of 
characteristic $0$. Let $C$ be an octonion  $R$--algebra such 
that $q_C$ is isometric to a Pfister quadratic form.
Then all $R$--isotopes of $C$ are trivial. 
\end{Prp}

The proof requires a preliminary statement  of the same flavour as Lemma \ref{lem_aniso1}.

\begin{Lma}\label{lem_tony} Let $A$ be an integral domain with fraction field
$K$. Assume that  $2$ is invertible
in $A$. Let $(M_1,q_1)$ and $(M_2,q_2)$ be regular quadratic $A$--forms
and consider
the $A[t^{\pm 1}]$--quadratic form $q= q_1 \perp t q_2$ defined on
$M_1 \otimes_A A[t^{\pm 1}] \oplus M_2 \otimes_A A[t^{\pm 1}]$.
We assume that $q_1$ and $q_2$ are $K$--anisotropic.
Then
$$
\bS_{q_1}(K)= \bS_q\bigl( K[t^{\pm 1}] \bigr) \hbox{\enskip and \enskip}
\bS_{q_1}(A)= \bS_q\bigl( A[t^{\pm 1}] \bigr).
$$
\end{Lma}

 \begin{proof}
 We claim that we have
 $$
(I) \qquad \qquad \qquad \bS_q\bigl( K((t)) \bigr) \subseteq M_1
\otimes_A K[[t]] \oplus  M_2 \otimes_A K[[t]]
 $$
 and
 $$
(II)  \qquad \qquad  \bS_{q}\bigl(K((t^{-1})) \bigr)\subseteq M_1
\otimes_A K[[t^{-1}]] \oplus
  M_2 \otimes_A K[[t^{-1}]] .
 $$
 
 We start with (I). Let $m=m_1+m_2 \in M\otimes_A K((t))$ such that
 $q(m)=1$. For $i=1,2$, we write $m_i= t^{n_i} \underline{m}_i$ with $\underline{m}_i \in M_i \otimes_A K[[t]]^*$. (Recall that
 $K[[t]]^*= K[[t]] \setminus tK[[t]]$.)
 
 If $m_1=0$, then $1= q( t^{n_2} \underline{m}_2)=  t^{2 n_2+1}
q_2(\underline{m}_2)$.
 Since $q_2$ is $K$-anisotropic, $q_2(\underline{m}_2)$ is a unit and we
get a contradiction. Thus $m_1$ is non-zero.
 If $m_2=0$, then $1= q( t^{n_1} \underline{m}_1)=  t^{2 n_1} q_1(\underline{m}_1)$.
  Since $q_1$ is $K$-anisotropic, $q_1(\underline{m}_1)$ is a unit and
we get $n_1=0$, whence $m_1 \in M_1 \otimes_A K[[t]]^*$.
  If $m_1$ and $m_2$ are both non-zero, then
 \[1= q( t^{n_1}m_1+t^{n_2} m_2)= t^{2 n_1} q_1(\underline{m}_1) +  t^{2
n_2+1} q_2(\underline{m}_2)\]
 where $q_1(\underline{m}_1)$ and $q_2(\underline{m}_2)$ are both units
(for the evaluation at $t=0$).
 For evaluation reasons the only possibility is that $n_1=0$ and $n_2 \geq 0$.
In all cases, the inclusion (I) holds.

 To show (II), the case $m_1=0$ and $m_2=0$ are similar to the corresponding cases in (I).
 Assume that $m_1$, $m_2$ are both non-zero and write
 $m_i= t^{-n_i} \underline{m}_i$ with
  $\underline{m}_i \in M_i \otimes_A K[[t^{-1}]]^*$
  for $i=1,2$. We have  $$
 1= q( t^{-n_1}m_1+t^{-n_2} m_2)= t^{-2 n_1} q_1(\underline{m}_1) +
t^{-2 n_2+1} q_2(\underline{m}_2)
 $$
  where $q_1(\underline{m}_1)$ and $q_2(\underline{m}_2)$ are both units
(for the evaluation at $\infty$).
 Once again for evaluation reasons, we have $n_1 = 0$ and $2n_2 -1 \geq 0$
so that $n_2 \geq 0$.
 The inclusion (II) holds.

 Consider now an element $x= m_1 + m_2 \in \bS_q\bigl( K[t^{\pm 1}]
\bigr)$
 where
 $m_i \in {M_i \otimes_A K[t^{\pm 1}]}$ for $i=1,2$. The inclusions (I) and (II)
 yield that $m_i \in M_i \otimes_A K$ for $i=1,2$.
 Furthermore the equation $1= q_1(m_1)+t q_2(m_2)$ implies that
 $q_1(m_1)=1$ and $q_2(m_2)=0$. The anisotropy of $q_{2,K}$ ensures
 that $m_2=0$. Thus $m=m_1$ comes from $\bS_{q_1}(K)$ as desired.

 The first statement is then established and the second statement
 readily follows.
 \end{proof}

 \begin{proof}[Proof of Proposition \ref{prop_laurent}]
 Our assumption is that $q_C$ is a Pfister form with entries $\langle\langle u,v,w \rangle \rangle$
 where $u,v, w\in R^*= k^* t_1^\Z \dots t_n^\Z$. Up to using the multiplicativity properties of $q_C$ and 
 the action of $\GL_n(\Z)$ on $R$, we may assume that 
 $q_C$ is isometric to a quadratic form of one of the following shapes:
 
 \smallskip
 
(i) $u,v,w \in k^*$.

 \smallskip

(ii) $u,v \in k^*$ and $w= w_n t_n$ with $w_n \in k^*$.

 \smallskip
 
  (iii) $u \in k^*$, $v=v_{n-1} t_{n-1}$, $w=w_{n} t_{n}$ with $v_{n-1}, v_{n} \in k^*$. 
 
\smallskip

  (iv) $u=u_{n-2} t_{n-2}$, $v=v_{n-1} t_{n-1}$, $w=w_{n} t_{n}$ with $u_{n-2}, v_{n-1}, v_{n} \in k^*$.

 \smallskip
 
Since (i) includes the hyperbolic case, we may assume in cases (ii), (iii) and (iv)  that 
$q_C$ is anisotropic over the field $k(t_1, \dots, t_n)$, and even over $k((t_1))\dots ((t_n))$ 
by  \cite[Cor. 7.4.3]{GP}.

 \smallskip

\noindent{\it Case (i):} Let $p$ be a $3$-Pfister form such that 
$p \otimes_k R \simeq q_C$. Let $C_p$ be an octonion
$k$--algebra whose norm is $p$. Such an algebra is unique up to isomorphism.
Then $C$ is a an isotope of the octonion $R$--algebra $C_p \otimes_k R$.
If $p$ is split (resp. anisotropic),  Corollary \ref{cor_zorn} (resp. Cor. \ref{cor_aniso}) shows that
$C$ is isomorphic to $C_p \otimes_k R$.

\smallskip

\noindent{\it Case (ii), (iii) and (iv) :} We apply Lemma \ref{lem_tony} with $A= k[t_1^{\pm 1},\dots, t_{n-1}^{\pm 1}]$, $t=t_n$ and
$q_1= \langle \langle u, v\rangle \rangle$, $q_2= w_n \, \langle \langle u,v \rangle\rangle$. It follows that
$$
\bS_{\langle \langle u, v\rangle \rangle}(k[t_1^{\pm 1},\dots, t_{n-1}^{\pm 1}])= \bS_{q_C}(R).
$$
Let $a \in \bS_{\langle \langle u, v\rangle \rangle}(k[t_1^{\pm 1},\dots, t_{n-1}^{\pm 1}])$. Then the isotope $C^{a,\overline{a}}$ is trivial by Corollary \ref{cor_cube}.(2).
We conclude that all isotopes of $C$ are trivial.
 \end{proof}

 \begin{Rk}
  (a) The characteristic zero hypothesis is probably superfluous, it should work
  in odd characteristic as well.

  \smallskip
  
  \noindent (b) Over $k[t^{\pm 1}]$, quadratic forms are classified and in particular diagonalizable \cite[Lemma 1.2]{Pa}.
  Our statement permits then to classify  octonion  $k[t^{\pm 1}]$--algebras and to recover Pumpl\"un's classification 
  \cite{Pu}. 
  
  \smallskip
  
  \noindent (c) Assume that the field $k$ is algebraically closed of characteristic zero.
  In this case, the so-called ``loop octonion algebras'' (that is whose underlying torsor is
  a loop torsor) over  $R=k[t_1^{\pm 1}, \dots ,t_n^{\pm 1}]$
  are classified \cite[Cor. 11.2]{GP} and their norms are Pfister forms.
  The statement shows that those octonion algebras have no non-trivial isotopes.
 \end{Rk}

\appendix

\section{Review on quadratic forms}

\subsection{Definition} In this work a quadratic module is a pair $(M,q)$ 
consisting of a
finitely generated and projective $R$--module $M$ and a quadratic form $q\co M
\to R$. For a quadratic module $(M, q)$ we will use the following notation
and terminology: $b_q \co M \times M \to R$ is the polar form of $q$, defined
 by $b_q (m_1, m_2) = q(m_1 + m_2) - q(m_1) - q(m_2)$ and often
    abbreviated by $b=b_q$. It gives rise to the $R$--linear map
\begin{equation} \label{quadfo-ch}
 b_q \chv \co M \to M\chv= \Hom_R(M,R), \quad b_q\chv(m)\, (n)= b_q(m,n). 
\end{equation}
The radicals of $b_q$ and $q$ are \begin{align*}
       \rad(b_q) & = \big\{m \in M : b_q(m,M) = 0 \big \}= \Ker (b_q\chv), \\
       \rad(q) &= \big\{ m\in M : q(m) = 0 = b_q(m, M) \big\} \, \subseteq \, 
\rad(b_q).
\end{align*}
For $S\in \Ralg$ the quadratic form $q_S \co M_S \to S$ is the  unique
quadratic form satisfying $q_S(m\ot 1_S) = q(m) \ot 1_S$. 

We call a quadratic form $q$ {\em regular\/} if $b_q\chv$ is an isomorphism of
$R$--modules, and {\em non-singular} if $\rad(q_K) = \{0\}$ for all fields
$K\in \Ralg$. We list some useful facts.

\smallskip

$\bullet$ Let $a\in R$. The ``$1$--dimensional'' form $\langle a
    \rangle \co R \to R$, defined by $r \mapsto ar^2$, is non-singular if and 
only if $
    a\in R^*$; it is regular if and only if $2a\in R^*$. 

\smallskip

$\bullet$ A regular quadratic form is non-singular, and if
    $2 \cdot 1_R \in R^*$ then these two properties coincide.

\smallskip

$\bullet$ If $q$ is regular resp.\ non-singular, then so is $q_S$
    for all $S\in \Ralg$.

\smallskip

$\bullet$ The hyperbolic form on $R^{2n}$, defined by $q_{\rm
    hyp} (r_{-n}, \ldots, r_{n}) = \sum_{i=1}^n r_i r_{-i}$, is regular and
    non-singular.

\smallskip

$\bullet$ If $(M,q)$ is a quadratic module where $M$ has constant
    even rank, then $q$ is regular if and only if $q$ is non-singular. 

\medskip

Following Knus \cite[III.3]{Kn}, we denote by $\Disc(R)$ the (commutative) group 
of isometry classes of non-singular quadratic forms of rank one over $R$ (also 
called discriminant modules) where the product arises from the tensor product. 
Since $\bmu_2$ is
the orthogonal group of the form $x \to  x^2$, the yoga of descent shows that 
the group $\Disc(R)$ is isomorphic to $H^1_\fppf(R, \bmu_2)$ (ibid, III.3.2).

\subsection{Quotients of unit spheres by the antipodal relation}
We consider the unit sphere $\bS_q$ of $q$. This is an affine  
$R$--scheme such  that for any $R$--ring $S$,
$\bS_q(S) = \bigl\{ m \in M_S \, \mid \,  q(m) = 1 \bigr\}$.
 We are also interested in the affine $R$--scheme $\widehat{\bS}_q$ defined by 
 $\widehat{\bS}_q(S) = \bigl\{ m \in M_S \, \mid \,  q(m) \in S^* \bigr\}$.
 We have a closed immersion $\bS_q \subset \widehat{\bS}_q$.

\begin{Lma}\label{lem_sphere} 

\begin{enumerate}
 \item The sphere ${\bS}_q$ is faithfully flat over $R$ 
and $\widehat{\bS}_q$ is smooth over $R$. 
 \item If $M$ is of constant even  rank $2r \geq  2$, then ${\bS}_q$ is smooth 
over $R$ of relative dimension
$2r -1$.
\end{enumerate}
\end{Lma}

\begin{proof}
 (1) The map ${\bS}_q \to \Spec(R)$ admits a splitting and hence is faithfully 
flat.
 The $R$--scheme $\widehat{\bS}_q$ clearly satisfies the lifting criterion of 
smoothness; it is
 then formally smooth, hence smooth since $\widehat {\bS}_q$ is of finite 
presentation.

 \smallskip
 
 \noindent (2) Using (1), the fibre-wise smoothness criterion 
\cite[$_4$.17.8.2]{EGA4} reduces the proof to the case
 of an algebraically closed field $k$. In this case
 $\bS_q$ is isomorphic  to the affine hypersurface $\sum_{i=1}^r x_iy_i=1$, 
 which is smooth by the Jacobian criterion. 
\end{proof}

 The $R$--group scheme $\bmu_2$  (resp.\ $\mathbb{G}_m$) acts on $\bS_q$ (resp.\ 
$\widehat\bS_q$) by homothety.
 We come to  the quotient scheme $\ubS_q = \bS_q/ \bmu_2 = \widehat{\bS}_q/ 
\mathbb{G}_m$,
 that is the quotient of the unit sphere by the antipodal relation \cite[ 
VIII.5.1]{SGA3}. We remind the reader
that the quotient map $f_q :  \bS_q \to  \ubS_q$ is finite, faithfully flat and 
that $\bS_q$ is an affine
$R$--scheme. Since the 
morphism $\widehat f_q : \widehat{\bS}_q \to {\ubS}_q$ is smooth,
we see that $ \ubS_q$ is smooth over $R$. We
are interested in the characteristic map $\varphi: \bS_q(R) \to  H^1_\fppf( R, 
\bmu_2) \simeq  \Disc(R)$.
For each $x \in \ubS_q(R)$,  $\widehat f_q^{-1}(x)$ defines an invertible 
$R$--module $\calL_x$
which is an $R$--submodule of $M$.    
    
\begin{Lma}
(1)   Let $x \in  \ubS_q(R)$. Let $ S/R$  be a flat cover such that $x$
lifts to some element $m \in \bS_q(S)$. Then $\calL_x \otimes_R S=  Sm$.

\smallskip

\noindent (2) The restriction of $q$ to $\calL_x$ is non-singular and 
$\varphi(x) = [( \calL_x, q)] \in \Disc(R)$.

\end{Lma}

\begin{proof}
 (1) Obvious.
 
 \smallskip
 
\noindent (2) We have to compare the $\bmu_2$--torsor $\bE_x= f^{-1}(x)$  with 
the 
$\bmu_2$--torsor $\bE$ defined
by $\bE(S) = \mathrm{Isom}_S\bigl( (R; 1)S, (\calL_x, q )\bigr)$
for each $R$--ring $S$.
Let $S$ be a flat cover of $R$ such that $\bE_x(S) \not = \emptyset$. There 
exists
$m \in  \bS_q(S)$ such that $f_q(m) = x_S$. By (1), we have $\calL_x  \otimes_R 
S = S m$
and an isomorphism $u_m : (R, 1)_S \simlgr (\calL_x, q )$, $1 \mapsto m$.
 This isomorphism is $\bmu_2$--equivariant; by faithfully flat descent it gives 
rise to an
 isomorphism $\bE_x \simlgr \bE$ of $\bmu_2$--torsors over $R$.
\end{proof}

\section{A Smoothness Condition}

\begin{Lma}\label{lem_smooth} Let $R$ be a Dedekind ring. Let $\gX$ be a
$R$--group scheme of finite presentation,
equidimensional with smooth connected fibres. Then $\gX$ is smooth.
\end{Lma}

\begin{proof} 
We denote by $K$ field of fractions of $R$. 
Let $\widetilde \gX$ be the schematic closure of $\gX_K$ in 
$\gX$ \cite[$_2$.2.8.1]{EGA4}. This is a closed $R$--group scheme of $\gX$ which is 
flat. Let $d$ be the dimension of the fibres of $\gX/R$.
Let $s \in \Spec(R)$. According to the upper semi-continuity of dimensions of fibers 
\cite[VI$_B$.4.1]{SGA3},
the $\kappa(s)$--scheme $\widetilde \gX_s$ has dimension $\geq d$.
On the other hand,    $\widetilde \gX_s$  is a closed subscheme of
$\gX_s$, whereby  $\widetilde \gX_s$ is of dimension $d$. But
 $\gX_s$ is smooth connected so that  $\widetilde \gX_s = \gX_s$.
 It follows that $\gX= \widetilde \gX$, as $\widetilde \gX$ is flat.
The fibre-wise smoothness criterion \cite[$_4$.17.8.2]{EGA4} enables us to 
conclude that $\gX$ is smooth over $R$.  
\end{proof}

\end{document}